\documentclass[a4paper,1pt]{amsart}
\usepackage[foot]{amsaddr}

\usepackage[margin=2.7cm]{geometry}
\usepackage[colorlinks,allcolors=cyan]{hyperref}
\usepackage{graphicx,color, amsfonts,amsmath}
\usepackage{amssymb}

\usepackage{textgreek}

\usepackage[titletoc]{appendix}
 
\usepackage{systeme}
\usepackage{graphicx}
\usepackage[]{units}%
\usepackage{color}%
\usepackage{mathrsfs}
\usepackage{bm}

\usepackage{float}
\usepackage{amsmath,amsfonts}
\usepackage{isomath}
\usepackage{amsthm}
\usepackage{amssymb}
\usepackage{algorithm,algorithmic}
\usepackage{array}
\usepackage{booktabs}
\usepackage{multicol}
\usepackage{multirow}
\usepackage{tabularx}
\usepackage{subfigure}
\usepackage{cite} 
\usepackage{stackrel}

\usepackage{graphicx}

\usepackage{soul}
\usepackage{amssymb,amsthm,mathtools}
\usepackage[nameinlink,noabbrev]{cleveref}

\crefname{assumption}{assumption}{assumptions}

\newcommand{\vertiii}[1]{{\left\vert\kern-0.25ex\left\vert\kern-0.25ex\left\vert #1 
    \right\vert\kern-0.25ex\right\vert\kern-0.25ex\right\vert}}

\newcommand{\llrrparen}[1]{
  \left\{\mkern-6mu\left\{#1\right\}\mkern-6mu\right\}}

\usepackage[dvipsnames]{xcolor}

\newcommand{\blue}{\color{black}} 
\newcommand{\blueX}[1]{{\color{black}#1}}
\newcommand{\greenX}[1]{{\color{black}#1}}
\newcommand{\red}{\color{black}}
\newcommand{\magenta}{\color{black}}
\newcommand{\magentaX}[1]{{\color{black}#1}}

\usepackage{cancel}
\usepackage{changes}

\newcommand{\EK}[1]{{\color{black}#1}}

\newcommand{\TODO}[1]{{\color{black}#1}}
\newcommand{\green}[1]{{\color{black}#1}}
\newcommand{\redhot}[1]{{\color{black}#1}}

\newtheorem{thm}{Theorem}[section]
\newtheorem{cor}[thm]{Corollary}
\newtheorem{lem}[thm]{Lemma}
\newtheorem{ass}[thm]{Assumption}

\theoremstyle{definition}
\newtheorem{defn}[thm]{Definition}
\theoremstyle{remark}
\newtheorem{rem}[thm]{Remark}
\numberwithin{equation}{section}

\begin{document}
\title[{\MakeLowercase{hp}-version arbitrarily shaped boundary elements dG method: Stokes systems}]{
{\MakeLowercase{hp-}}version analysis for arbitrarily shaped elements \EK{on the boundary} discontinuous Galerkin method for {{S}}tokes systems}
\author{Efthymios N. Karatzas\textsuperscript{1
}}
\address{\textsuperscript{1}School of Mathematics, Aristotle University of Thessaloniki, Thessaloniki 54124, Greece
.}
\email{ekaratza@math.auth.gr
}
\subjclass[2000]{Primary}%
\keywords{\EK{A}rbitrarily shaped elements, 
          discontinuous Galerkin finite element method, 
          hp-version stability,
          a-priori estimates,
         fluid dynamics.
          }%
\date{\today} 

\thanks{}%
\subjclass{}%

\begin{abstract}
In the present work, we examine and analyze 
 an \TODO
 {$hp$-version  interior penalty discontinuous Galerkin finite element method for the numerical approximation of a steady fluid system}
 \TODO
 {on 
 computational meshes consisting of polytopic elements on the boundary.
 } 
This approach is based on the 
discontinuous Galerkin method, 
enriched by  
 arbitrarily shaped elements 
 techniques 
 \TODO
 {as 
 has been introduced in \cite{Cangiani2022hpVersionDG}}. 
In this framework, \TODO
{and employing extensions of trace, Markov-type, and $H^1$/$L^2$-type inverse estimates to arbitrary element shapes}, we examine a stationary Stokes fluid system \TODO
{enabling the proof} of the inf/sup condition and the {$hp$}- a priori error estimates, 
while we investigate the optimal convergence rates numerically. 
%
This approach recovers and integrates the flexibility and superiority of the discontinuous Galerkin 
methods for fluids whenever geometrical deformations are taking place
\TODO
{by degenerating   the edges, facets, of the polytopic elements only on the boundary,} combined with  the efficiency of the $hp$-version techniques based on arbitrarily shaped elements 
\TODO
{without requiring any mapping from a given reference frame}.

\end{abstract}
\maketitle
\section{Introduction}
%
%
 Recent years have shown scientists' interest \EK{significantly} focused on 
the context of Galerkin finite element methods
. 
This effort has given birth to new methods based on general-shaped elements which  arise computational complexity reduction, 
like mimetic finite difference methods \cite{DaVeiga14mimetics}, virtual element methods {\cite{DAVEIGA13virtualelements}}, 
various discontinuous Galerkin approaches such as interior penalty Galerkin methods \cite{Georgoulis17}, and hybridized discontinuous Galerkin \cite{CockburnHDG09,DIPIETRO15}, 
which are very attractive and used by the engineering \EK{and mathematics} community. 
\TODO
{We continue by reporting more works related to discontinuous Galerkin (dG) methods, similar finite element frameworks, and advances, $h$- or $hp$- version, \TODO
{fluid and Stokes systems} \TODO
{related} literature. We also introduce the skeleton of the present work.}
\TODO
{Hence,} other 
approaches have involved non-polynomial
approximation spaces like polygonal and other generalized finite element methods, 
\cite{SuTa04,BeTe10}. 
%
We \EK{refer} to \cite{Georgoulis17} 
for admissible polygonal/polyhedral element shapes for which the general interior penalty 
discontinuous Galerkin method (IP-dG), appears both stable and convergent while generalizes under mild assumptions 
the validity of 
standard approximation results, such as inverse estimates, best approximation estimates, and extension theorems. 

In a $p$-version Galerkin framework achieving  
exponential convergence, for smooth partial differential equation  problems defined  on generally curved domains using 
isoparametrically mapped elements, we recall \cite{MURTI86,MURTI88}, while for non-linear maps on element patches  that 
 are used to  represent domain geometry we refer to \cite{Melenk2002hpFiniteEM,Schwa98}. Although in both cases\EK{, as the polynomial order increases,} the aforementioned mapping appears very costly and/or difficult to construct and 
implement in practice. 
\TODO
{For Stokes flow systems, in \cite{T02,BJK90}, 
%
an $hp$-discontinuous Galerkin approximation that shows better stability properties than the
corresponding conforming ones is examined with finite element triangulation not required to be conforming employing discontinuous pressures and velocities\EK{,} while it is
defined on the interfaces between the elements 
involving the jumps of the velocity and the average of
the pressure. 
The work of \cite{GRW05} also, describes a family of dG finite element methods formulated and analyzed for Stokes and Navier-Stokes problems introducing the good behavior of the inf-sup and optimal energy estimates for the velocity
and 
pressure. In addition,  this method
can treat a finite number of non-overlapping domains with non-matching grids
at interfaces.}
%
%
%
%
In \cite{CKSS02, SST03,Ka24}, 
 Stokes system local discontinuous Galerkin methods for a class of shape regular meshes with hanging nodes is investigated, 
%
%
as well as, several mixed discontinuous Galerkin approximations with their a priori error estimates. 
%
%
In \cite{BE09}, 
%
a 
discontinuous Galerkin (dG) approach to simulations on complex-shaped domains, using trial and test functions defined on a structured grid with essential boundary conditions imposed weakly,  
where the discretization allows the number of unknowns to be independent of the complexity of the domain.
%
%
\cite{M12} concerns 
an unfitted dG method proposing to discretize elliptic interface problems, 
where $h$- and $hp$- error estimates and convergence rates are proved. 
%
The authors of \cite{WC14}, 
treat an unfitted dG method 
 for the elliptic interface problems, based on a variant of the local dG method, obtaining the optimal convergence for the exact solution in the energy norm and its flux  in the $L^2$ norm. 
In \cite{BEFI11} 
%
%
an unfitted discontinuous Galerkin method for transport processes in complex domains in porous media problems is examined, 
%
 allowing finite element meshes that are significantly coarser than those required by standard conforming finite element approaches. 
Further, in 
\cite{ERW16} 
%
an advection problem is developed 
based on an unfitted discontinuous Galerkin approach where the surface is not explicitly tracked by the mesh which means the method is 
 flexible with respect to geometry 
efficiently capturing the advection driven by the evolution of the surface without the need for a space-time formulation, back-tracking trajectories or streamline diffusion. 
Finally, in  \cite{EMNS19} 
a
linear transport equation on a cut cell mesh using the upwind discontinuous Galerkin method with piecewise linear polynomials and  with a method of lines approach is presented employing explicit time-stepping schemes, regardless of the presence of cut cells. 

In addition, various classes of fitted and unfitted mesh methods for interface
or transmission problems may be seen as 
generalized concepts of mesh elements, as well as, 
 several unfitted finite element methods have been proposed in recent years, indicatively we mention the unfitted boundary finite element methods \cite{BaElliot87} 
and  immersed finite element methods 
\cite{MaSco17_2}. 
\TODO
{More extensively, an optimally convergent method of 
 fictitious type domains 
  avoiding the numerical integration on cut mesh elements for a Poisson system has been introduced in \cite{L19},
%
%
while in
\cite{HaHa02}
%
a method for the finite element solution of the elliptic interface problem, using an approach due to Nitsche is proposed allowing discontinuities internal to the elements approximating the solution across the interface. 
In addition, from a reduced basis for unfitted mesh methods point of view, evaluating the fixed background mesh used in immersed and unfitted mesh methods, parametrized Stokes and other flow systems have 
managed to be solved using a unified reduced basis presenting the flexibility of such methods in geometrically parametrized Stokes, Navier-Stokes, Cahn-Hilliard systems as in \cite{KBR19,KSASR2019,KSNSRa2019,KSNSRb2019,KaRo21}. 
%
%
For an immersed interface method for discrete surface representations employing accurate jump conditions evaluated along  interface representations using  projections, one could see \cite{KB18}, 
and for a ghost fluid method coupled with a volume of fluid method employing an exact Riemann solver
 \cite{BG14}.  
For a fictitious domain finite element method, well suited for elliptic problems posed in a domain given by a level-set function without requiring a mesh fitting the boundary can be found in
\cite{DL19}. 
%
%
%
%
}
The early work of \cite{P1972} 
handles 
the flow of a viscous incompressible fluid containing immersed boundaries which move with the fluid and exert forces on the fluid, 
 and in \cite{WFC13} 
a finite difference scheme with ghost cell technique is used to study viscous fluid flow 
with internal structures. 
In \cite{PHO16}, 
%
%
a conservative cut-cell Immersed Boundary method 
with sub-cell resolution is analyzed. 

We extend the literature with
\cite{HNPK13}, 
%
where a high-order hybridizable dG method for solving elliptic interface problems in which the solution and gradient are nonsmooth because of jump conditions across the interface and it is endowed with several distinct characteristics. 
%
%
 \cite{DWXW17} contains an unfitted hybridizable dG method mesh method for the Poisson interface problem constructing an ansatz function in the vicinity of the interface 
with an appropriate choice of flux for treating the jump conditions, 
designed through a 
piecewise quadratic Hermite polynomial interpolation with post-processing via a standard Lagrange polynomial interpolation. 
These approaches usually employ penalization on the boundary interface and/or weak enforcement of the boundary conditions and data 
usually supported by 
a level set geometry description, 
\cite{StaRo03}. 
%
%

In the present work, we investigate the applicability of the interior-penalty discontinuous Galerkin method discretizing steady Stokes flow cases onto meshes with boundaries considering general, essentially arbitrarily shaped element shapes in the sense of \cite{Cangiani2022hpVersionDG}, which allow attaining 
smaller errors compared with other competitive finite element methods \TODO
{e.g. comparing with the unfitted dG approaches of \cite{AreKarKat22}}, 
fact verified numerically in Section \ref{section6}. Furthermore, our analysis allows for curved element shapes 
without the use of any non-linear elemental maps. 
\TODO
{We use extended} 
\TODO
{$hp$-  trace and inverse estimates to the arbitrary shape of boundary elements 
and we prove the inf-sup stability of the method in proper, to the prescribed method, norm. A priori error bounds for the resulting method are given under very weak 
 assumptions.
} Numerical experiments are also presented,
indicating the efficiency of the proposed framework.

This manuscript is structured as follows: we start with the Stokes flow model problem and the necessary preliminaries in Section \ref{section2}. The various components of the interior penalty stabilized arbitrary boundary elements dG discretization 
are discussed in subsection \ref{22} and we recall trace inverse estimates that are pivotal in the proof of the stability of IP-dG methods employed with 
the constants that appear and affect several inverse, stability, and error estimates. Approximation results needed for the analysis of the method are collected in Section \ref{section3}. Section \ref{section4} is devoted to stability estimates and the derivation of the discrete inf--sup condition, followed by a--priori error estimates in Section \ref{section5}. Concluding, the aforementioned analysis is numerically verified with tests in Section \ref{section6} that \EK{depict} the \EK{optimal} theoretical $hp$-convergence rates and the $hp$-accuracy of the method.
%

It is noted that this work \EK{determines an 
 approach} where many expensive penalization terms can be omitted and appear beneficial  especially, for fluid systems where additional penalization is needed.  
 These results \TODO
 {are}  
 relevant to 
 the $hp$-version framework \TODO
 {and geometrical challenges} 
 and verify that  
there is no loss of stability and accuracy, as traditionally appears e.g. in cases with excessive and/or  insufficient penalization that typically results in accuracy loss. 
%
%
%
%
%
%
%
Further, we report that the theoretical tools presented are adapted to a Nitsche-type formulation. 
The inf-sup stability result of the method in a Stokes-like norm is proved 
on  $hp$-approximation that will also
 lead to an error bound, \TODO
 {an a priori error estimate} and optimal convergence rates.
The theoretical developments presented regarding stability and a-priori error analysis of IP-dG methods considering 
general arbitrarily shaped 
elements
, to the best of our knowledge  \TODO
are new for Stokes systems
%
%
.
%
\begin{figure}[ht] 
\centering
\includegraphics[trim={4cm 3.8cm 4cm 3.6cm},clip,scale=0.27] {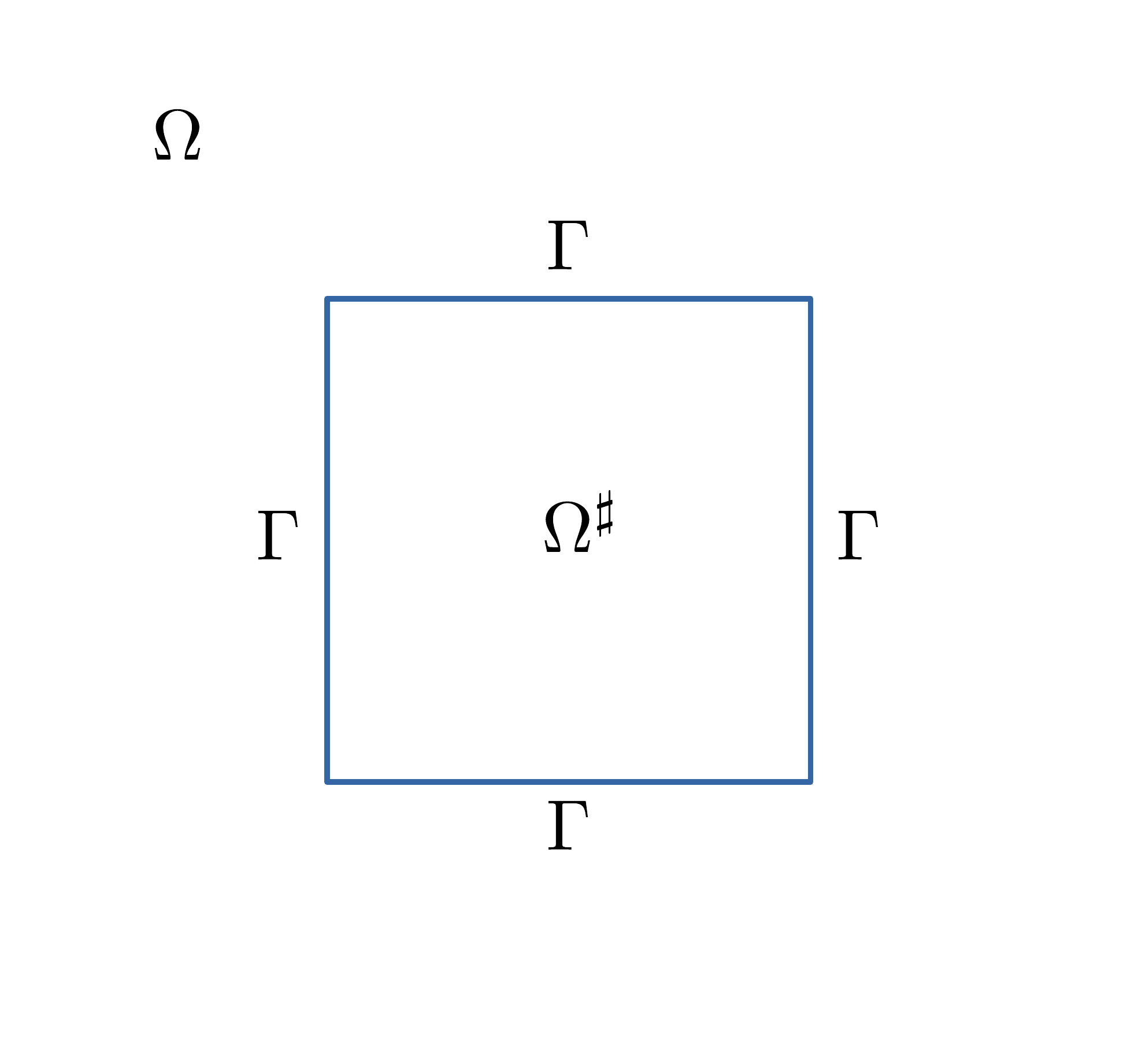}\hskip22.5pt
\includegraphics[trim={4cm 3.8cm 4cm 3.6cm},clip,scale=0.27] {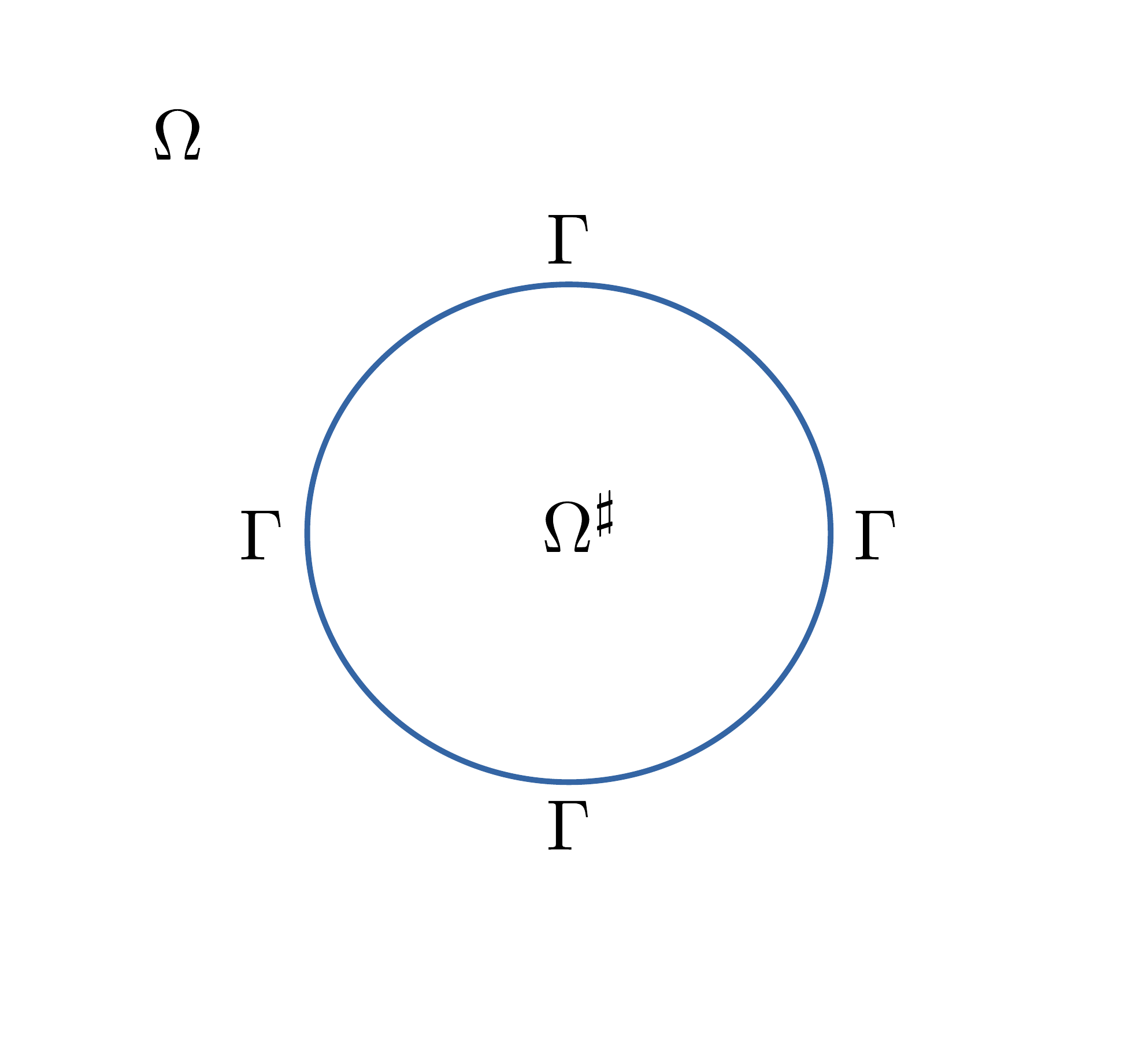}
\caption{{\greenX{Example geometries $\Omega^{\sharp}$ and their boundaries $\Gamma$}}. 
}
\label{geometry}
\end{figure}
\begin{figure}[ht]
\includegraphics[scale=0.55] {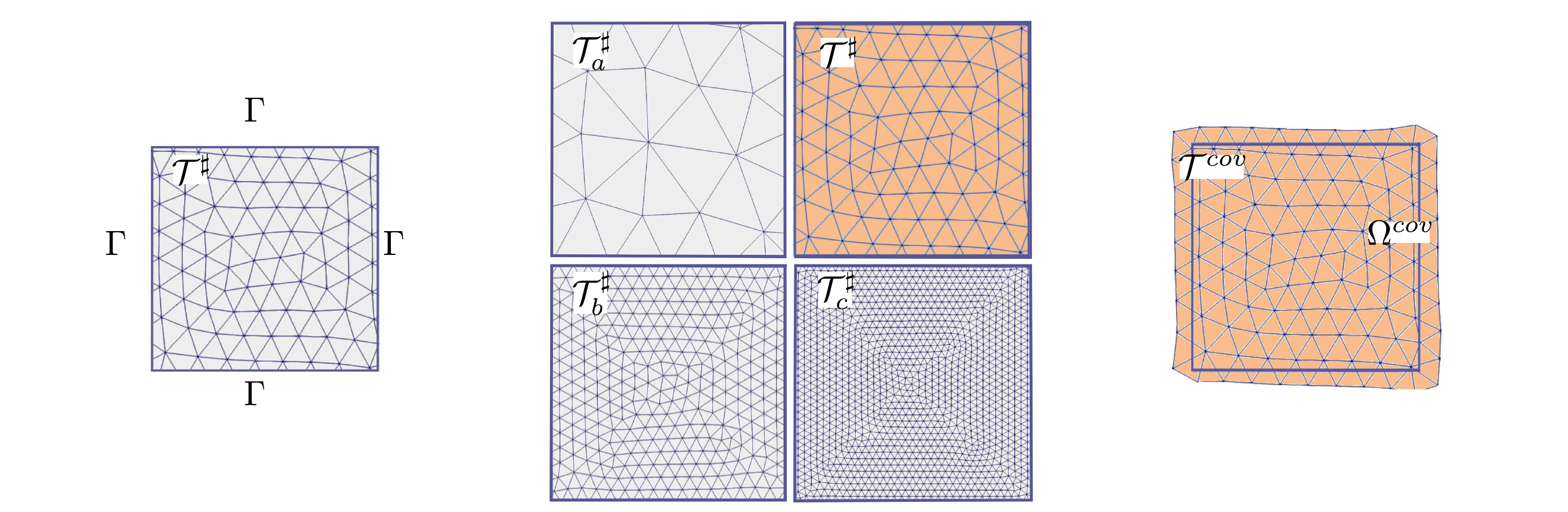} \\
\includegraphics[scale=0.55] {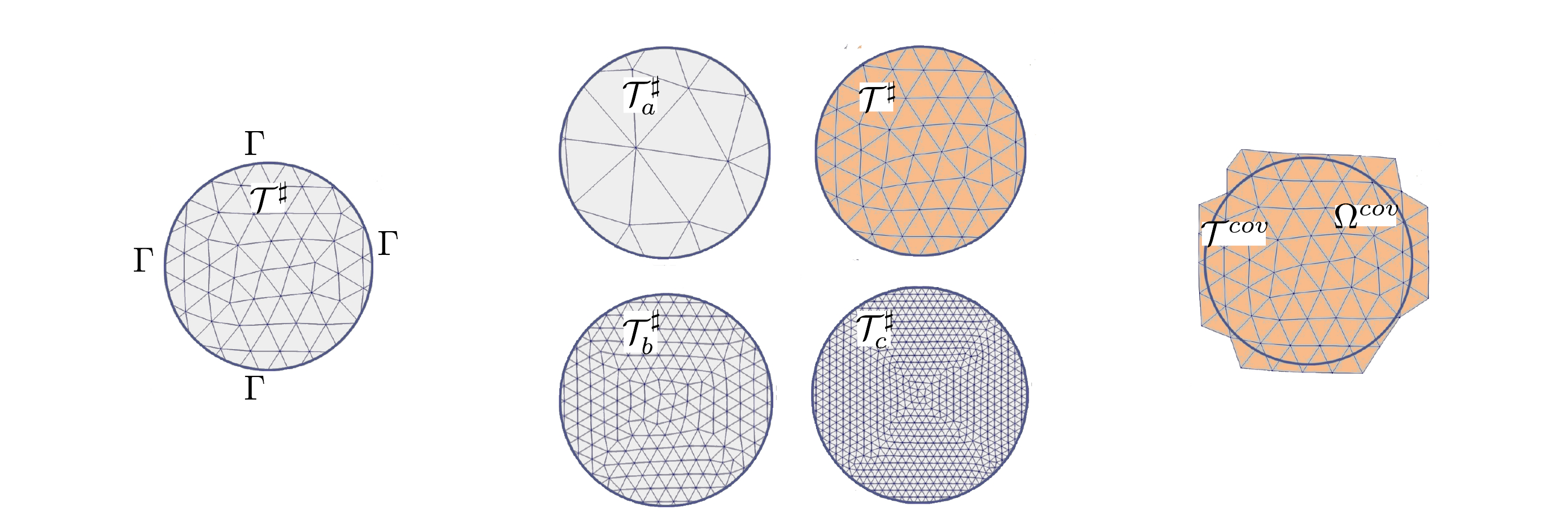} 
\caption{\greenX{Meshes} $\mathcal{T}^{\sharp}$ based on  arbitrarily-shaped boundary elements, the covering \greenX{meshes} $\mathcal{T}^{cov}$ from Definition \ref{def:covering_domain}, examples of refined  \greenX{Figure \ref{geometry}'s  geometries} $\Omega^\sharp $ tessellations $\mathcal{T}_a^\sharp$, $\mathcal{T}_b^\sharp$, $\mathcal{T}_c^\sharp$ and the covering \greenX{domains}, $\Omega^{cov}$.}
\label{mesh_geometry}
\end{figure}
\section{
Model  problem
} \label{section2}
\subsection{Problem formulation} The steady Stokes equations for an incompressible viscous fluid confined in an open, bounded domain $\Omega^\sharp \subset \mathbb{R}^d$ ($d=2, 3$) with Lipschitz boundary $\Gamma=\partial \Omega^\sharp$  can be expressed in the form
 \begin{eqnarray} \label{The Stokes Problem}
 - \Delta  \mathbf{u}+\nabla p=\mathbf{f}
 \text{ and } \nabla \cdot \mathbf{u}=0  \quad \text{in $\Omega^\sharp$}, 
\quad \text{ with } \quad \mathbf{u}=0 \quad \text{on $\Gamma$}.
\end{eqnarray}
Here $\mathbf{u}=(u_1, \dots, u_d):\Omega^\sharp  \rightarrow  \mathbb{R}^d$ ($d=2,3$) and $p:\Omega^\sharp \rightarrow  \mathbb{R}$ denote the velocity and pressure fields, and $\mathbf{f} \in \left[L^2(\Omega^\sharp)\right]^d$ is a forcing term. 
 Since the pressure is determined by (\ref{The Stokes Problem}) up to an additive constant, we assume $\int_{\Omega^\sharp}p dx= 0 $ to uniquely determine $p$. Hence, in the following, we will consider for pressure the standard space
$
L_0^2(\Omega^\sharp):=\left\{q \in L^2(\Omega^\sharp):\int_{\Omega^\sharp}q \,dx=0 \right\}
$ 
of square-integrable functions with zero average over $\Omega^\sharp$.

%
Defining for all $\mathbf{u}, \mathbf{v} \in  V^\sharp :=[H_0^1(\Omega^\sharp)]^d$ and $p\in Q^\sharp:= L_0^2(\Omega^\sharp)$  the bilinear forms 
\begin{align}
a(\mathbf{u}, \mathbf{v})= \int_{\Omega^\sharp}\nabla \mathbf{u} : \nabla \mathbf{v}\,d\bm{x},  \quad 
b(\mathbf{v}, p)= -\int_{\Omega^\sharp}p  \nabla \cdot \mathbf{v}\,d\bm{x}, \label{forms_cont}
\end{align} 
a weak solution to (\ref{The Stokes Problem}) is a pair   $(\mathbf{u}, p) \in [H_0^1(\Omega^\sharp)]^d \times          L_0^2(\Omega^\sharp)=V^\sharp\times Q^\sharp$, such that
\begin{equation}\label{weak}
A(\mathbf{u}, p ; \mathbf{v}, q)=\int_{\Omega^\sharp}\mathbf{f} \cdot \mathbf{v}\,d\bm{x}, \ \ \text{for all test functions} \ \   (\mathbf{v}, q) \in V^\sharp\times Q^\sharp,
\end{equation}
with 
$$
A(\mathbf{u}, p ; \mathbf{v}, q)=  a(\mathbf{u}, \mathbf{v}) +b(\mathbf{u}, q)+b(\mathbf{v}, p).
$$
The well--posedness of (\ref{weak}) is standard \cite{DiPE12}.

\subsection{Arbitrarily shaped  discontinuous Galerkin method on the boundary}\label{22}
Implementation of arbitrarily shaped boundary elements discontinuous Galerkin method for the discretization of (\ref{weak}) \TODO
{relies on {a 
\TODO
{covering} domain} ${\Omega}^{cov}$ which contains the \EK{true} geometry} $\Omega^{\sharp}$, see Figures \ref{geometry}, \ref{mesh_geometry}. \TODO
{Let ${\mathcal{T}^{cov}}$ be the corresponding covering} shape-regular 
mesh of $\Omega^{cov}$ 
and $\mathcal{T}^\sharp$  is the mesh corresponding to $\Omega^\sharp$. The  {\itshape{active}} mesh 
$$
{\mathcal{T}^\sharp}=\left\{\mathcal{K}\cap \Omega^\sharp ; \text{ for all } \mathcal{K} \in \TODO
{{\mathcal{T}^{cov}}} \text{ with } \mathcal{K}\cap \Omega^\sharp\neq \emptyset\right\}
$$  
is the  \TODO
{minimal submesh of $ \mathcal{T}^{cov}$} which covers $\Omega^\sharp$ \TODO
{employing polytopic, e.g. polygonal boundary interface elements} and is \TODO
{actually} {\itshape{fitted}} to its boundary $\Gamma$: {we allow mesh boundary elements $K \in \mathcal{T}^\sharp$ which are
arbitrarily shaped and with very general interfaces. In the present work, numerical experiments consider  boundary elements as \EK{Lipschitz} curved elements and with only curved facet the one that coincides to the corresponding part of the boundary $\Gamma$, see Figures \ref{geometry}, \ref{mesh_geometry} 
or the more general case of Figure \ref{fig:curved_boundary_elements}. However, one could also employ general interfaces with neighboring elements, \cite{Georgoulis17,Dong19,dong2022hp}}. 

Finite element spaces for $\mathbf{u}$ and $p$ will be built upon the domain $\Omega^\sharp
=\bigcup_{K \in \mathcal{T}^\sharp}K$ which corresponds to  $\mathcal{T}^{\sharp}$.  
\EK{ The mesh skeleton $
\bigcup_{K \in \mathcal{T}^\sharp} \partial K$ --including the curved boundary facets-- is subdivided into the internal part
$$
\mathcal{F}_{int}^{\sharp}=\left\{F=K^{+}\cap K^{-}: K^{+}, K^{-} \in \mathcal{T}^{\sharp} \text{ and } F\nsubseteq\Gamma\right\}
$$
which actually denotes the set  of interior faces in the active mesh and the boundary part  
${\mathsf{\Gamma}}\equiv \partial\Omega^\sharp$. 
}
\EK{We denote by
 $h_K=diam(K)$ the \itshape local \normalfont mesh size for $K\in \mathcal{T}^\sharp$ boundary elements and $h_F=\min\left\{h_{K^{+}}, h_{K^{-}}\right\}$ for $F=K^{+}\cap K^{-} 
$.}

We choose to enforce boundary conditions at $\Gamma$ to be weakly satisfied through Nitsche's method. We highlight that we do not employ coercivity recovery techniques applied over the \TODO
{covering 
domain $\Omega^{cov}$}\EK{, e.g.} \TODO
{by means of additional ghost penalty terms which act on the gradient} jumps \EK{--usually higher order--} in the boundary zone; see, for instance, \cite{BH2012, BuHa14_III, MLLR12, KBR19}. Instead, the ${\mathcal{T}}^{\sharp}$ coercivity is ensured \EK{following the} approach of \cite{Georgoulis17,Dong19}. 

To define the arbitrarily \EK{shape discontinuous} Galerkin discretization for the Stokes problem (\ref{weak}), 
we employ the element-wise discontinuous  polynomial  finite elements for pressure and velocity spaces of order $\mathfrak{p}\geq 1$:
\begin{align*}
 V_h^\sharp \equiv S^\mathfrak{p}_{\mathcal{T}^\sharp,{\mathbf{u}}_h}
 &:=\left\{w_h\in \left(L^2(\Omega^{\sharp})\right)^d: w_h|_{K}\in \left(\mathcal{P}^{\mathfrak{p}}(K)\right)^d, K \in \mathcal {T}^{\sharp}\right\} \quad (d=2,3)\\
Q_h^\sharp \equiv S^\mathfrak{p}_{\mathcal{T}^\sharp,p_h}
&:=\left\{w_h\in L_0^2(\Omega^\sharp): w_h|_{K}\in \mathcal{P}^{\mathfrak{p}-1}(K), K \in \mathcal {T}^{\sharp}\right\}.
\end{align*} 
%
The broken Sobolev space \blueX{$H^\mathfrak{p}
(\Omega^\sharp, \mathcal{T}^\sharp )$}, with
respect to the subdivision $\mathcal{T}^\sharp$ up to \blueX{composite order 
$\mathfrak{p}
$},
 is defined as
$$ \blueX{H^\mathfrak{p}
}(\Omega^\sharp, \mathcal{T}^\sharp) = \{w \in L^2 (\Omega^\sharp) : w|_K \in \blueX{H^{\mathfrak{p}
}(K)} \
\forall K \in \mathcal{T}^\sharp \}.$$

 {It is important to mention that 
whenever {the notation} $\nabla v$ is used for functions that lay in the discontinuous Galerkin space, {i.e. $w\notin H^1({\Omega}^\sharp)$}, will correspond to the broken gradient, such that, $(\nabla w)|_K=\nabla(w|_K)$ for all $K\in\mathcal{T}^\sharp$. 
So, the broken gradient $\nabla_{\mathcal{T}^\sharp} w$ of a function $w \in L^2 (\Omega^\sharp)$ with $w|_K \in
H^1 (K)$, for all $K \in \mathcal{T}$, is defined element-wise by $(\nabla_{\mathcal{T}^\sharp} w )|_K := \nabla(w|_K )$.} When $F \subset \mathsf{\Gamma}$, it is $\llrrparen{w} =[\![w]\!] = w$ and $\llrrparen{\mathbf{w}}=[\![\mathbf{w}]\!] = \mathbf{w}
$. 
 The same applies to the broken divergence operator $\nabla\cdot w$ defined element-wise.
Moreover, recall the definition 
$$ 
\llrrparen{w}:=\frac{1}{2} \left(w^{+}+w^{-}\right), \quad \quad  \llrrparen{\mathbf{w}}:=\frac{1}{2} \left(\mathbf{w}^{+}+\mathbf{w}^{-}\right),
$$
of the {\em average} operator $\llrrparen{\cdot}$ across an interior face $F$ for $w$, $\mathbf{w}$   scalar and vector-valued functions  on $\mathcal{T}^{\sharp}$ respectively, where $w^{\pm}$ (resp. $\mathbf{w}^{\pm}$) are the traces of $w$ (resp. $\mathbf{w}$) on $F=K^{+}\cap K^{-}$ from the interior of $K^{\pm}$. More precisely, with $w^{\pm}(\mathbf{x})=\lim_{t\rightarrow 0^{+}}w(\mathbf{x}\pm t\bf{n_F})$ for $\mathbf{x} \in F$ and  $ { \bf n_F}$  we denote the outward-pointing unit normal vector to $F$ and  with $  \mathbf{n}_{\Gamma}$ the outward unit normal to the boundary $\Gamma$. The  {\em jump} operator $ [\![\cdot ]\!]$ across $F$  is defined respectively by 
$$ 
 [\![w]\!]:=w^{+}-w^{-}
 , \quad \quad   [\![\mathbf{w} ]\!]:=\mathbf{w}^{+}-\mathbf{w}^{-}
 .
$$
{
\begin{figure}
\begin{minipage}{\textwidth}
\centering 
{\footnotesize{(i)}}\hskip1.0pt
\includegraphics[width=0.465\textwidth]{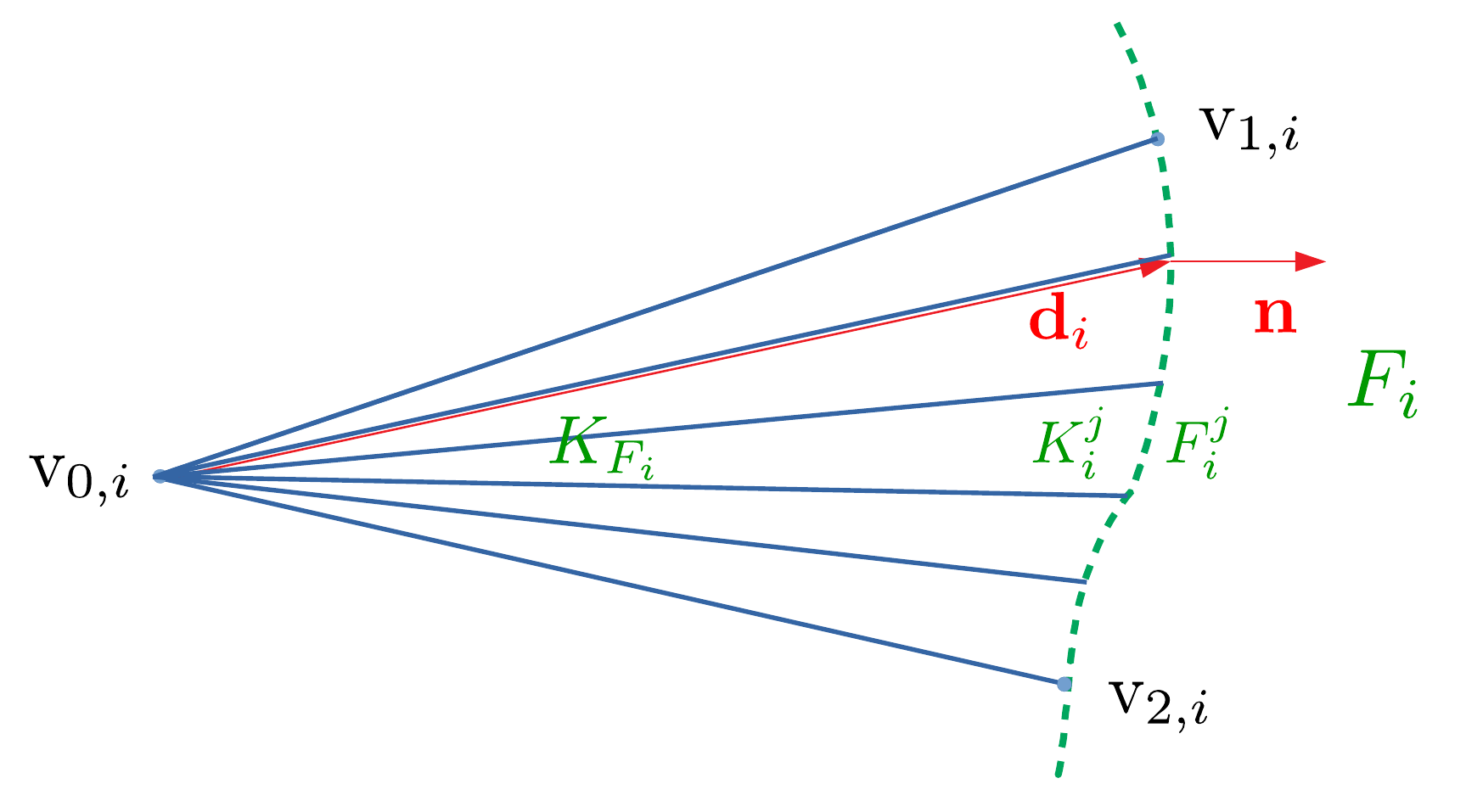}~{\footnotesize{(ii)}}\hskip1.0pt~\includegraphics[width=0.465\textwidth]{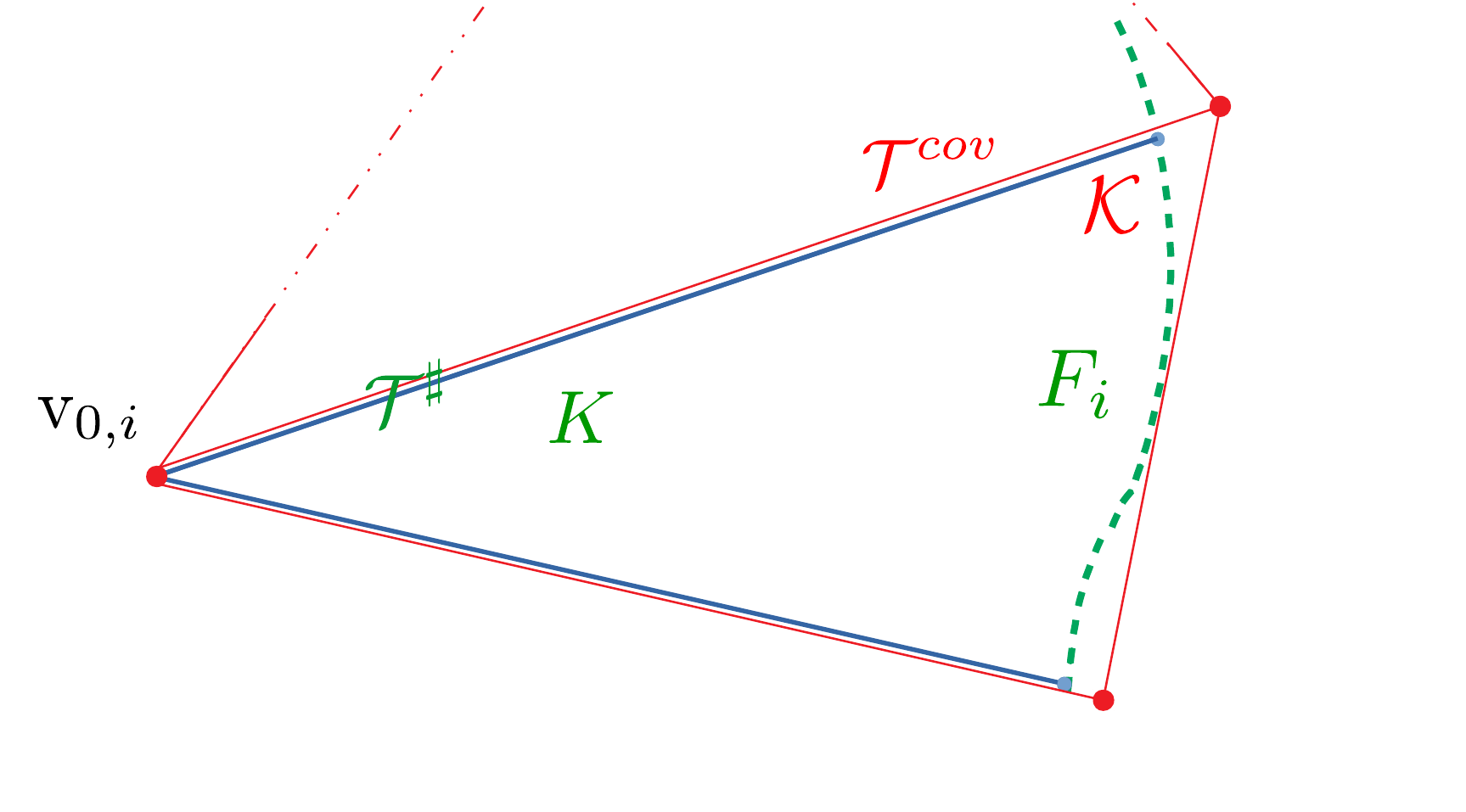}
\end{minipage}
\caption{(i) Curved boundary elements for $d = 2$ with one curved face and  vertices $\mathbf{v}_{k,i}$ and the unit outward normal vector to $F_i$ at $x \in F_i$ where $K_{F_i}$ is star-shaped and 
 (ii) the covering mesh $\mathcal{T}^{cov}$, and the mesh $\mathcal{T}^\sharp$ corresponding to the truth geometry with arbitrary shape boundary elements.   
 }
\label{fig:curved_boundary_elements}
\end{figure}
} 
\EK{We} are now ready to formulate a discrete counterpart of (\ref{weak}) through \EK{a discontinuous} Galerkin method. The symmetric interior penalty discretizations of the diffusion term and the pressure--velocity coupling in (\ref{forms_cont}) lead to the bilinear forms
\begin{align*}
a_h(\mathbf{u}_h, \mathbf{v}_h) = \int_{\Omega^\sharp} \nabla \mathbf{u}_h: \nabla \mathbf{v}_h \,d\bm{x} - \sum_{\redhot{F \in \mathcal{F}^\sharp_{int} \cup {\mathsf{\Gamma}}}
}\int_{F
}\left(\llrrparen{\nabla \mathbf{u}_h}\cdot \bf{n_F} [\![ {\mathbf{v}}_h  ]\!] + \llrrparen{ \nabla \mathbf{v}_h}\cdot \bf{n_F} [\![ \mathbf{u}_h  ]\!] \right)d\gamma 
\\
- \int_{\Gamma}\mathbf{u}_h\nabla \mathbf{v}_h \cdot  \mathbf{n}_{\Gamma}\,d\gamma -\int_{\Gamma}\mathbf{v}_h \nabla \mathbf{u}_h \cdot \mathbf{n}_{\Gamma}\,d\gamma +
\int_{\Gamma} {\magenta{\sigma}} \mathbf{u}_h\mathbf{v}_h\,d\gamma + 
\sum_{\redhot{F\in \mathcal{F}^\sharp_{int}\cup {\mathsf{\Gamma}}}}
\int_{F
}{\magenta{\sigma}}[\![ \mathbf{u}_h  ]\!] [\![ \mathbf{v}_h  ]\!]\,d\gamma , 
\\
b_h(\mathbf{v}_h, p_h)= - \int_{\Omega^\sharp}p_h\nabla \cdot \mathbf{v}_h\,d\bm{x}+\sum_{\redhot{F\in \mathcal{F}^\sharp
_{int}\cup {\mathsf{\Gamma}}}}
\int_{F
} [\![ \mathbf{v}_h  ]\!]\cdot \mathbf{n_F} \llrrparen{  p_h}d\gamma +\int_{\Gamma} \mathbf{v}_h \cdot \mathbf{n_{\Gamma}}  p_h\,d\gamma,
\end{align*}
where $\sigma>0$ is the 
 discontinuity-penalization function in $ \redhot{L^\infty(\mathcal{F}^\sharp _{int}\cup {\mathsf{\Gamma}}}
)$ that affects the stability of the method as well as the approximation quality and will be investigated below.
This symmetric interior penalty parameter  in the definition of $a_h(\cdot, \cdot)$ will be  sufficiently large  in a sense that will be made precise later, see Lemma \ref{a_coerc} and its proof below.

We note that the latter formulation's disadvantage is that it is not well defined for $H^1(\Omega^\sharp)$ regularity, e.g. traces of functions defined in $L^2(\Omega^\sharp)$ are not well defined in $\mathcal{F}_{int}^{\sharp}$. The latter issue affects the terms $\llrrparen{\nabla(w)}$, and $\nabla{w}\cdot \mathbf{n}_F$ in the sense that they are not well defined in $H^1(\Omega^\sharp)$.
This \EK{causes the need of} additional regularity while the Galerkin orthogonality cannot be derived explicitly. In order to achieve optimal a priori error estimates, under the presence of terms such $\llrrparen{\nabla w}|_F$, ${\nabla w}\cdot\mathbf{n}_F|_F$ which \EK{may} involve ${\llrrparen{\nabla(w-\pi_\mathfrak{p} w) }}_F$ and ${\nabla (w-\pi_\mathfrak{p} w)}\cdot\mathbf{n}_F|_F$, where $\pi_\mathfrak{p} w$ is an operator $\pi_\mathfrak{p} : H^{l_K}(
\mathcal
{K}) \to \mathcal{P}_\mathfrak{p}(
\mathcal
{K})$ an approximation of $w$ for $l_k\ge 0$ \EK{will be} introduced in Lemma \ref{lem:hp_projection_error_bounds} and it is estimated optimally. We note 
 at this point that  the $H^1(F)$ semi-norm for an $hp$- a priori approximation error bound 
would require $W^{1,\infty}$ norm error bounds which also require further regularity, see for more details the work of \cite{Georgoulis17}. To avoid the latter \EK{issue} we employ proper bilinear form extensions. In particular we introduce the orthogonal $L^2$-projection in the FEM space $S^{\mathfrak{p}}_{\mathcal{T}^\sharp,\cdot}$, e.g. $\mathbf{\Pi}_{L^2}:(L^2(\Omega^\sharp))^d \to (S^{\mathfrak{p}}_{\mathcal{T}^\sharp,\cdot})^d$
concluding in the variational form:
\begin{align*}
\widetilde{a}_h(\mathbf{u}_h, \mathbf{v}_h) = \int_{\Omega^\sharp} \nabla \mathbf{u}_h: \nabla \mathbf{v}_h\,d\bm{x} - \sum_{\redhot{F \in \mathcal{F}^\sharp_{int} \cup {\mathsf{\Gamma}}}}
\int_{F}\left(\llrrparen{  {\magenta{\mathbf{\Pi_{L^2}}(}}\nabla \mathbf{u}_h)}\cdot \bf{n_F} [\![ {\mathbf{v}}_h  ]\!] + \llrrparen{ {\magenta{\Pi_{L^2}(}}\nabla \mathbf{v}_h)}\cdot \bf{n_F} [\![ \mathbf{u}_h  ]\!] \right)d\gamma \\
- \int_{\Gamma}\mathbf{u}_h{\magenta{\mathbf{\Pi_{L^2}}(}}\nabla \mathbf{v}_h ) \cdot  \mathbf{n}_{\Gamma}\,d\gamma -\int_{\Gamma}\mathbf{v}_h  {\magenta{\mathbf{\Pi_{L^2}}(}}\nabla \mathbf{u}_h) \cdot \mathbf{n}_{\Gamma}\,d\gamma+ 
\int_{\Gamma} {\magenta{\sigma}} \mathbf{u}_h\mathbf{v}_h\,d\gamma + 
\sum_{\redhot{F\in \mathcal{F}^\sharp_{int}\cup {\mathsf{\Gamma}}}}
\int_{F}{\magenta{\sigma}}[\![ \mathbf{u}_h  ]\!] [\![ \mathbf{v}_h  ]\!]\,d\gamma, \\
b_h(\mathbf{v}_h, p_h)= - \int_{\Omega^\sharp}p_h\nabla \cdot \mathbf{v}_h\,d\bm{x}
+\sum_{\redhot{F\in \mathcal{F}^\sharp_{int}\cup {\mathsf{\Gamma}}}
}
\int_{F} [\![ \mathbf{v}_h  ]\!]\cdot \mathbf{n_F}\llrrparen{p_h}d\gamma +\int_{\Gamma} \mathbf{v}_h \cdot \mathbf{n_{\Gamma}}  p_h\,d\gamma
\end{align*}
respectively. For future reference,  note that element-wise integration by parts in the previous forms yields the equivalent formulations
\begin{align}
\widetilde{a}_h(\mathbf{u}_h, \mathbf{v}_h)= - \int_{\Omega^\sharp} \nabla\nabla
\mathbf{u}_h \cdot  \mathbf{v}_h\,d\bm{x} + \sum_{\redhot{F\in \mathcal{F}^\sharp_{int}\cup {\mathsf{\Gamma}}
}}
\int_{F} [\![ {\magenta{\mathbf{\Pi_{L^2}}(}}\nabla \mathbf{u}_h)  ]\!] \cdot \mathbf{n_F} \llrrparen{\mathbf{v}_h}\,d\gamma
\notag\\
-\sum_{\redhot{F \in \mathcal{F}^\sharp_{int}\cup {\mathsf{\Gamma}}
}
}\int_{F} \llrrparen{  {\magenta{{\mathbf{\Pi_{L^2}}}(}}\nabla {\mathbf{v}}_h)}\cdot {\bf{n_F}} [\![ {\mathbf{u}}_h  ]\!] \, d\gamma \notag  
-\int_{\Gamma}{{\mathbf{u}}_h{\magenta{{\mathbf{\Pi_{L^2}}}(}}\nabla {\mathbf{v}}_h)\cdot  {\mathbf{n}_{\Gamma}}} \, d \gamma 
\\ 
+ \int_{\Gamma} {\magenta{\sigma}} \mathbf{u}_h\mathbf{v}_h\,d\gamma 
+ \sum_{\redhot{F\in \mathcal{F}^\sharp_{int}\cup {\mathsf{\Gamma}}
}
}\int_{F} {\magenta{\sigma}} [\![ \mathbf{u}_h  ]\!] [\![ \mathbf{v}_h  ]\!] \,d\gamma , \label{a_alt}
\end{align}
\begin{align}
b_h(\mathbf{v}_h, p_h)= \int_{\Omega^\sharp}\mathbf{v}_h\cdot \nabla p_h\,d\bm{x} 
- \sum_{\redhot{F\in \mathcal{F}^\sharp_{int}\cup {\mathsf{\Gamma}}
}
}\int_{F} \llrrparen{\mathbf{v}_h }\cdot \mathbf{n_F} [\![ p_h  ]\!]\,d\gamma,\label{b_alt}
\end{align}
which will be useful for asserting the consistency of the method. 

Using the aforementioned weak formulation, an arbitrarily shape boundary elements discontinuous Galerkin method for (\ref{weak}) now reads as follows: Find $(\mathbf{u}_h,p_h)\in V^\sharp_h\times Q^\sharp_h$, such that 
\begin{equation}\label{cutdg}
A_h(\mathbf{u}_h, p_h ; \mathbf{v}_h, q_h)
=L_h(\mathbf{v}_h, q_h), \ \ \text{for all} \ \   (\mathbf{v}_h, q_h) \in V_h^\sharp\times Q_h^\sharp.
\end{equation}
The bilinear and linear forms $A_h$ and $L_h$ are defined by
\begin{align}
A_h(\mathbf{u}_h, p_h ; \mathbf{v}_h, q_h)=  \widetilde{a}_h(\mathbf{u}_h, \mathbf{v}_h) + b_h(\mathbf{u}_h, q_h)+b_h(\mathbf{v}_h, p_h)
, 
\,\,\, \text{ and } \,\,\,
L_h(\mathbf{v}_h, q_h)= \int_{\Omega^\sharp} \mathbf{f}\cdot \mathbf{v}_h \,d\bm{x}
. \label{A} 
\end{align}
We report that in the right-hand side $L_h({\mathbf{v}_h,q_h})$ we have omitted the zero Nitsche boundary terms, as well as, $\widetilde{a}_h(\mathbf{u}_h, \mathbf{v}_h) = {a}_h(\mathbf{u}_h, \mathbf{v}_h)$ when $\mathbf{u}_h, \mathbf{v}_h \in  V_h^\sharp$. 
\section{Preliminaries
} 
\label{section3}
Next, {\green{we define the discontinuity penalization parameter $\sigma : \mathsf{\Gamma} \cup  \mathcal{F}^\sharp _{int} \to \mathbb{R}$, the }}  standard Sobolev norms and semi--norms on a domain $\mathcal{X}$ for $s\in \mathbb{N}$ will be denoted by $ \|\cdot \|_{s, \mathcal{X}}$ and $|\cdot|_{s, \mathcal{X}}$\EK{,} respectively, omitting the index in case $s=0$.
The a--priori error bounds for the proposed dG method will be proved with respect to the following mesh-dependent norms:
 \begin{eqnarray}
 \vertiii{\mathbf{v}}^2 =  \|\nabla \mathbf{v} \|^2_{\Omega^\sharp} +  \|{\magenta{\sigma}^{1/2}}
\mathbf{v} \|^2_{\Gamma}
 +
\EK{\sum_{{{{F\in\mathsf{\Gamma}}
}}}}
 {\red{ \|{\magenta\mathfrak{p}^{-1}}h_{\EK{F}}^{1/2}
 \nabla \mathbf{v}
 \cdot \mathbf{n}_{\Gamma} \|^2_{{\EK{F}}} }}
\qquad\qquad\qquad\qquad\qquad\qquad
\nonumber
\\ +  \sum_{F\in \mathcal{F}^\sharp_{int}}  \|{\green \sigma^{1/2}}
[\![\mathbf{v} ]\!] \|^2_{\EK{F}
} 
 + \EK{\frac{1}{2}}{\red{\sum_{{{{K\in
\EK{\mathcal{T}^{cov}}
}}}} \|{\magenta{\mathfrak{p}^{-1}}}h_K^{1/2}\nabla \mathbf{v} |_{T}\cdot \mathbf{n}_{T} \|_{\EK{\partial K}
}^2}},
\nonumber \\
  \vertiii{p}^2 =     \| p \|^2_{\Omega^\sharp} 
  +  
\EK{\sum_{{{{F\in\mathsf{\Gamma}}
}}}}
\| {\magenta{\mathfrak{p}^{-1}}}h_{\EK{F}}^{1/2}p \|^2_{\EK{F}} + \sum_{F\in \mathcal{F}^\sharp_{int}}  \|{\magenta{\mathfrak{p}^{-1}}}h_F^{1/2}[\![p]\!] \|^2_{\redhot{F
}}
   + \EK{\frac{1}{2}}{\red{\sum_{{{{K\in
\EK{\mathcal{T}^{cov}}
}}}}  \|{\mathfrak{p}^{-1}}h_K^{1/2}p \|^2_{\EK{\partial K}
}}}, 
\nonumber
 \end{eqnarray}
 \normalsize
and   $\vertiii{(\mathbf{v},p)}^2  =  \vertiii{\mathbf{v}}^2+\vertiii{p}^2$.
To investigate stability, \green{since some terms of the above terms dominate related to others}, we will also make use of the following norms in $\Omega^\sharp$  for the discrete velocity and  pressure approximations, \green{e.g. for the velocity norm the third and fifth terms appear similar $hp$- behavior with the first 
 term and for pressure the second, third and fourth term with the first term. For this reason, we also update and define the norms}:  
 \begin{eqnarray}
\vertiii{\mathbf{v}}^2_{\EK{V^c}} = { \|\nabla \mathbf{v} \|^2_{\EK{\Omega^{cov}}}} +  \|{\magenta{\sigma}^{1/2}}
 \mathbf{v} \|^2_{\Gamma} + \sum_{F\in \mathcal{F}^\sharp_{int}}  \|{\green{\sigma}^{1/2}}
[\![\mathbf{v} ]\!] \|^2_{\redhot{F
}}, \,\,\,\text{ and } \,\,\,
 \vertiii{p}^2_{\EK{Q^{c}}} = { \| p \|^2_{\EK{\Omega^{cov}}}}, 
 \end{eqnarray} 
{\green{while $\vertiii{(\mathbf{v},p)}^2_{{\EK{V^{c}}},{\EK{Q^{c}}}} =  \vertiii{\mathbf{v}}_{{\EK{V^{c}}}}^2+\vertiii{p}^2_{{\EK{Q^{c}}}}$.}}
\green{We underline,  that in the following, and for completeness we treat all the aforementioned terms showing this equivalence.}
 
The following section is devoted to useful trace and inverse estimates, which have been proved in \cite{Georgoulis17,Dong19} 
 and they will form the basis to prove a--priori error estimates of the proposed method. 
\subsection{{Inverse estimates (trace and $H^1$--$L^2$)}}
It is easy to derive the estimates with respect to the  norms $\vertiii{\cdot }$ and $\vertiii{\cdot}_{{\EK{V^{c}}} \text{ or }{\EK{Q^{c}}}}$, namely,
\begin{equation}\label{norm_est1} 
\vertiii{\mathbf{v}}\le C_{{\EK{V^{c}}}} \vertiii{\mathbf{v}}_{{\EK{V^{c}}}}, \quad \vertiii{p} \le C_{{\EK{Q^{c}}}} \vertiii{p}_{{\EK{Q^{c}}}}.
\end{equation}
\begin{ass}\label{ass:basic_geometry}
For each element $K \in \mathcal{T}^\sharp$ with $K\cap \Gamma\neq \emptyset$, we assume that $K$ is a Lipschitz domain, and  $\partial K$ can be subdivided into mutually exclusive subsets $\{F_i\}^{n_K}_{i=1}$
characterized by the property that respective sub-elements $K_{F_i}\equiv K_{F_i} (\text{v}_{0,i}) \subset K$ there exist,
with $d$ planar faces meeting at one vertex $\text{v}_{0,i} \in K$, with $F_i \subset \partial K_{F_i}$: for
$i = 1, . . . , n_K$, we consider that (a) $K_{F_i}$  is star-shaped with respect to $\text{v}_{0,i}$, and
(b) $\mathbf{d}_i(x) \cdot \mathbf{n}_{F_i}(x) > 0$ for $\mathbf{d}_i (x) := x - \text{v}_{0,i}$, $x \in K_{F_i}$, and $\mathbf{n}_{F_i}(x)$ the respective unit outward normal vector to $F_i$ at $x \in F_i$.
It is also considered that the boundary $\partial K$ of each element $K \in \mathcal{T}^\sharp$, $K\cap \Gamma \neq \emptyset$ is
the union of a finite  number of closed $C^1$ surfaces.
\end{ass}
Both (a) and (b) assumptions, for the two-dimensional case, are visualized in Figure \ref{fig:curved_boundary_elements}. We notice that in the above  weak mesh assumption, the sub-domains $\{F_i \}^{n_K}_{i=1}$ are not required to coincide with the faces of the element $K$, namely, each $F_i$ may be part of a face or may include one or more faces of $K$, as well as, there is no requirement for $\{n_K\}_{K\in \mathcal{T}^\sharp, K\cap\Gamma \neq 0}$ to remain
uniformly bounded across the mesh. In particular Assumption \ref{ass:basic_geometry} states that the curvature of the collection of consecutive curved faces comprising $F_i$ cannot be arbitrarily large almost everywhere.
Moreover, with some small loss of generality, Assumption \ref{ass:basic_geometry} b) can be made stronger by adding the ingredient that it is possible to consider a fixed point $\text{v}_{0,i}$ such that there exists a global constant $c_{sh} > 0$, such that
\begin{align}\label{ineq:global}
\mathbf{d}_i(x) \cdot \mathbf{n}_{F_i}(x) \ge c_{sh} h_{K_{F_i}} 
\end{align}
see e.g. [24, 65].
We underline that (\ref{ineq:global}) does not imply shape-regularity of the $K_{F_i}$'s; in particular $K_{F_i}$'s with small $F_i$ compared to the remaining (straight) faces of
$K_{F_i}$ are acceptable. Such anisotropic boundary sub-elements $K_{F_i}$'s may be necessary
to ensure that each $K_{F_i}$ remains star-shaped when an element boundary's
curvature is locally large, see e.g., $K_{F_i}$ in Figure \ref{fig:curved_boundary_elements} and a collection of both shape-regular and anisotropic $K_{F_i}$'s. 
%
In general, $F_i$ is not required to be connected, although, by splitting $F_i$ to its connected
subsets, re-indexing the $F_i$'s to correspond to unique $K_{F_i}$, we can correspond
one $K_{F_i}$ to each ${F_i}$. 
{The aforementioned Assumptions \ref{ass:basic_geometry} are sufficient for the proof of the trace estimates as well as for the validity of the $H^1$--$L^2$
inverse estimate as in \cite{Georgoulis17}. 
}
{
\begin{lem} Let element $K \in \mathcal{T}^\sharp$ be a Lipschitz domain satisfying Assumption
\ref{ass:basic_geometry}. Then, for each $F_i \subset \partial K$ 
\EK{from Assumptions \ref{ass:basic_geometry}},
$i = 1, ..., n_K$, and for each $v \in \mathcal{P}^\mathfrak{p} (K)$, we have the inverse estimate:
\begin{align}\label{ineq:inv_est}
||v||^2_{F_i}\le\frac{(\mathfrak{p} + 1)(\mathfrak{p} + d)}{\min\limits_{x\in F_i}{(\mathbf{d}_i}\cdot \mathbf{n}_{F_i})}
||v||^2_ {K_{F_i}}.
\end{align}
\end{lem}
\begin{rem}
Inequality (\ref{ineq:inv_est}) is a function of $\text{v}_{0,i}$ defining $K_{F_i}$. 
The right-hand side can be minimized with a choice of an optimal $\text{v}_{0,i}$. We underline that under the stronger assumption (\ref{ass:basic_geometry}), one could derive the trace inverse estimate for star-shaped, shape-regular elements with piece-wise smooth boundaries: $||v||^2_{\partial K} \le C\frac{{\green{\mathfrak{p}}}^2}{h_K}||v||^2_K$.
\end{rem}
}
\begin{lem}
Let $K \in \mathcal{T}^\sharp$ be a Lipschitz domain satisfying Assumption \ref{ass:basic_geometry}
. Then,
for all $\varepsilon > 0$, we have the estimate
$$
||v||^2_{F_i} \le \frac{d + \varepsilon}{
{\min\limits_{x\in F_i}{(\mathbf{d}_i}\cdot \mathbf{n}_{F_i})}
}||v||^2_ {K_{F_i}} 
+ 
\frac{\max\limits_{x\in F_i}{|\mathbf{d}_i|_2^2}}{
{\varepsilon\min\limits_{x\in F_i}{(\mathbf{d}_i}\cdot \mathbf{n}_{F_i})}
}||\nabla v||^2_{K_{F_i}}
,
$$
for all $v \in H^1(\Omega^\sharp)$ and $i = 1, . . . , n_K$.
We note that summing over $i = 1, . . . , n_K$, under assumption $\mathbf{d}_i(x) \cdot \mathbf{n}_{F_i}(x) \ge c_{sh}h_{K_{F_i}}$ and that $h_{K_{F_i}} \sim h_K$ we take the estimate gives the classical trace estimate $||v||^2_{\partial K} \le C  (h^{-1}_K ||v||^2_K + h_K||\nabla v||^2_K )$.
\end{lem}
{
\begin{defn}\label{def:p-coverable}
An element $K \in \mathcal{T}^\sharp$ is said to be ${\green{\mathfrak{p}}}$-coverable with respect to ${\green{\mathfrak{p}}} \in \mathbb{N}$, if there exists a set of $m_K$ overlapping shape regular simplices $K_i$, $i = 1,...,m_K$, $m_K \in \mathbb{N}$, such that
\begin{align}
dist(K, \partial K_i ) < C_{\text{as}}\frac{\text{diam}(K_i)}{{\green{\mathfrak{p}}}^2}, \text{ and }
 |K_i| \ge c_{\text{as}}|K|
\end{align}
for all $i = 1, . . . , m_K$, where $C_{\text{as}}$ and $c_{\text{as}}$ are positive constants, independent of $K$ and $\mathcal{T}^\sharp$.
\end{defn}
}
\begin{lem}[\cite{Georgoulis17} ] \label{lem:Fi_to_K}
Let $K \in \mathcal{T}^\sharp$ Lipschitz satisfying Assumption \ref{ass:basic_geometry}. Then, for each $\text{v} \in \mathcal{P}^{\green{\mathfrak{p}}}(K)$, we have the inverse inequality 
\begin{align}\label{eq:Fi_to_K}
 ||v||^2_{F_i} \le C_{INV}({\green{\mathfrak{p}}}, K, F_i)\frac{({\green{\mathfrak{p}}}+1)({\green{\mathfrak{p}}}+d)|F_i|}{|K|}||v||^2_K,
\end{align}
with $C_{\text{INV}}({\green{\mathfrak{p}}}, K, F_i)$ to be if $K$ is ${\green{\mathfrak{p}}}$-coverable:
$
\min\{\mathcal{C}(K,F_i), {c_{as}^{-1} 2^{5d+1}{\green{\mathfrak{p}}}^{2(d-1)}}\}   
,
$
otherwise: 
\\
$\mathcal{C}(K,F_i)$,
with $c_{\text{as}} > 0$ as in Definition \ref{def:p-coverable}
and
$\mathcal{C}=\frac{|K|}{|F_i|\sup\limits_{\mathbf{v}_{0,i}\in K} \min\limits_{\mathbf{x}\in F_i}(\mathbf{d}_i\cdot\mathbf{n}_{F_i})}
$.
%
\end{lem}
%
%
%
After defining the covering domain $\bar{\Omega}^{cov}$ and considering the Assumption \ref{ass:max_card},  see e.g. \cite{Georgoulis17,Dong19}, 
 we interpolate a pair   $(\mathbf{u}, p) \in  [H^{2}(\Omega^\sharp) ]^d \times H^1(\Omega^\sharp)$ through a suitable interpolant of 
%
$
  [H^{\mathfrak{p}+1}(\mathrm{\Omega^\sharp}) ]^d \times H^\mathfrak{p}(\mathrm{\Omega^\sharp})$ -extensions of the functions $(\mathbf{u}, p)$  on $\Omega^{cov}
$.
{
\begin{defn}\label{def:covering_domain}
Given a mesh $\mathcal{T}^\sharp$, we define a covering $\mathcal{T}^{cov} = \{\mathcal{K}\}$ of $\mathcal{T^\sharp}$ to be a
set of open shape-regular $d$–simplices $\mathcal{K}$, such that for each $K\in\mathcal{T^\sharp}$, there exists a
$\mathcal{K}\in \mathcal{T}^{cov}$ with $K \subset\mathcal{K}$. For a given $\mathcal{T}^{cov}$, we define the covering domain $\bar{\Omega}^{cov} := \cup_{\mathcal{K}\in \mathcal{T}^{cov}} {\bar{\mathcal{K}}}$.
\end{defn}
 \begin{ass}\label{ass:max_card}
For a given mesh $\mathcal{T}^\sharp$, we postulate the existence of a covering
$\mathcal{T}^{cov}$, and of a global constant $\mathcal{O}_{\Omega^\sharp} \in \mathbb{N}$, independent of the mesh parameters, such
that
$$
\max\limits_{K\in\mathcal{T}}\text{card}\{K' \in \mathcal{T}: K'\cup\mathcal{K} \neq \emptyset, \mathcal{K} \in \mathcal{T}^{cov}  \text{ such that } K \subset \mathcal{K} \}\le \mathcal{O}_{\Omega^\sharp}.
$$
For such $\mathcal{T}^{cov}$, we further assume that $h_\mathcal{K} := \text{diam}(\mathcal{K}) \le C_\text{diam}h_K$, for all pairs
$K \in \mathcal{T}^{\sharp}$, $\mathcal{K} \in \mathcal{T}^{cov}$, with $K \subset \mathcal{K}$, for a global constant $C_{diam} > 0$, uniformly with
respect to the mesh size $h_K$.
\end{ass}
\redhot{The latter assumption provides 
the shape-regularity of the covering mesh $\mathcal{T}^{cov}$  \EK{--not though for the true $\mathcal{T}^\sharp$--}  \EK{ in the sense 
that there exists a positive constant $c$, independent of the mesh parameters, such that
$\forall K \in \mathcal{T}^{cov}\subset \mathcal{T}$,  $\rho_\mathcal{{K}} \ge c {h_\mathcal{K}}$ holds, with $\rho_\mathcal{{K}}$ denoting the diameter of the largest ball contained in $\mathcal{K}$}.} 
%
The aforementioned 
 will allow  the application of the standard $hp$-version approximation estimates on simplicial elements, see e.g., from \cite{Schwa98} that
on each $\mathcal{K}$ we can restrict the error over $K \subset \mathcal{K}$. However, it requires to extend \EK{properly} the exact solution $u$ onto $\Omega^{cov}$. 
In particular:
}{
\begin{thm}
 Let $\Omega^\sharp$ be a domain with a Lipschitz boundary. Then there
exists a linear extension operator $\mathfrak{E}: H^s(\Omega^\sharp) \to H^s(\mathbb{R}^d)$, $s \in \mathbb{N}_0$, such that $\mathfrak{E}v|_{\Omega^\sharp} = v$
and
$$
||\mathfrak{E}v||_{H^s(\mathbb{R}^d)} \le C_\mathfrak{E} ||v||_{H^s(\Omega^\sharp)}, $$
where $C_\mathfrak{E}$ is a positive constant depending only on $s$ and on $\Omega^\sharp$.
\end{thm}
}
We also recall from \cite{Georgoulis17} 
the $H^1$--$L^2$ inverse inequality for polynomials on a general curved element $K \in \mathcal{T}^\sharp$.
\begin{lem}
\magentaX{Let $K \in \mathcal{T}^\sharp$ satisfy Assumptions \ref{ass:basic_geometry}, \ref{ass:max_card} and (\ref{ineq:global})
. Then, for each $v \in \mathcal{P}^\mathfrak{p} (K)$, the inverse estimates
hold, 
 for $K$  $\mathfrak{p}$-coverable:
%
\begin{align}
  {||\nabla v||_K  }  { \le C \frac{{\mathfrak{p}}^2}{h_ K} ||v||_K}, \text{ and } 
\label{estimate:nabla_v}
\\
{||\nabla v \cdot \mathbf{n}_F||_{F}\le C'\frac{\mathfrak{p}^3}{h^{3/2}_K}
||v|| _K
   }\label{estimate:normal_der_v}
\end{align}
hold, 
and the constants $C$, $C'$ are dependent on the shape-regularity constant.
}
\end{lem}
\begin{proof}
\magentaX{The first estimate comes immediately from \cite{Georgoulis17}
will the second comes from the algebraic calculations: ${||\nabla v \cdot \mathbf{n}_F||_{F}\le C\frac{\mathfrak{p}}{h^{1/2}_K}||\nabla v||_{K} \le
{C}' \frac{{\mathfrak{p}}^2}{h_ K} 
\frac{\mathfrak{p}}{h^{1/2}_K}
||v|| _K =
{C}'
\frac{\mathfrak{p}^3}{h^{3/2}_K}
||v|| _K}$.
}
\end{proof}
\section{Stability estimates} \label{section4}
The fact that the discrete problem is well-posed follows by the inf--sup stability of the bilinear form $A_h
$ in the formulation \eqref{cutdg} with respect to the \EK{$\vertiii{\cdot}_{V^c
,Q^c
}$} norm.  We begin by investigating the properties of the separate forms which contribute to $A_h
$.  
A useful observation is that the form $\widetilde{a}_h(\cdot, \cdot)$, 
is continuous and coercive with respect to the norm \EK{$\vertiii{\cdot }_{V^c
}$. For this proof, we will use 
that the arbitrarily shaped boundary elements 
can be properly extended 
from the real 
domain $\Omega^{\sharp}$ to the 
covering one, 
$\Omega^{cov}$.
}
%
%
\begin{lem}
\label{ext}  
There are constants $C_{u
}, C_{p}>0$, depending only on the
shape-regularity and the polynomial order and not on the mesh or the location of the boundary, such that the following estimates hold: 
\begin{eqnarray}\label{ext_v}
& \|\nabla \mathbf{
\magentaX{v
_h}} \|^2_{\magentaX{\Omega
^{cov}}
}  \leq C_{u
}  \|\nabla \mathbf{v}_h \|^2_{\Omega^\sharp}
 \leq C_{u
}  \|\nabla \mathbf{v}_h \|^2_{\magentaX{\Omega
^{cov}}
 },  \quad \text{for all } \mathbf{v}_h \in V_h^\sharp,
\text{ and }&
\\
\label{ext_p}
& \| p_h \|^2_{\magentaX{\Omega
^{cov}}}  \leq C_{p}  \| p_h \|^2_{\Omega^\sharp}
 \leq C_{p}  \| p_h \|^2_{\magentaX{\Omega
 ^{cov}
 }},   \quad \text{for all } p_h \in Q_h^\sharp.&
\end{eqnarray}
\end{lem}
\begin{proof}
\EK{Considering also the  Assumption \ref{ass:basic_geometry}}, 
 we assume that there is an integer $N > 0$ such that for each element 
\EK{$K \in\mathcal{T}^{cov}$ with $K\cap\Gamma \ne \emptyset$} there exists an element 
$\EK{K'\in \mathcal{T}^{cov}}$  \EK{with $K'\cap \Gamma = \emptyset$}
and at most $N$ elements $\{ K \}^N_{i=1}$ such that $K_1 = K$, $K_N = K'$ and $K_i \cap K_{i+1} \in 
\EK{\partial _i \mathcal{T}^{cov}}$, $i = 
\EK{1}, . . . N-1$. 
 \EK{In particular, this means that the number of facets we need to cross so that we pass from the aforementioned element $K$ to $K'\subset \Omega^\sharp$ is bounded.}
Similar assumptions were made by \cite{AreKarKat22}, see also references therein,
which ensure that $\Gamma$ is reasonably resolved by 
 \EK{$\mathcal{T}^{cov}$}. 
\EK{For} the 
{first}
 inequality 
\EK{\eqref{ext_v}} 
we compose the norm over 
 \EK{$\Omega^{cov}$} into sums over \EK{internal} 
and \EK{boundary $\mathcal{T}^{cov}$ elements, $\|\cdot\|_{\mathcal{T}^{cov}} = \sum{\|\cdot\|}_{\mathcal{T}^{cov}\text{-boundary-}K\text{'}s} +\sum{\|\cdot\|}_{\mathcal{T}^{cov}\text{-internal }K\text{'}s} $}. 
Let $K_0 
$
 be 
\EK{a boundary element of $\mathcal{T}^{cov}$}. 
 Then, there exists a $K_N \subset 
\EK{\Omega^\sharp}$ and at most $N - 1$ \EK{$\mathcal{T}^{cov}$-boundary} elements $K_i 
$ and facets $K_{i-1} \cap K_i = F_i 
$ that has to be overtaken in order to go across from $K_0$ to $K_N$. Considering the aforementioned shape-regularity of the mesh, each facet \EK{corresponding to $\mathcal{T}^{cov}$-boundary elements} $F 
$ will only be involved in a finite number of such crossings. 
\EK{Additionally, let $v$ be 
a polynomial function 
{of order $\mathfrak{p}$} defined on 
{both} the 
{boundary} element 
{$K\in \mathcal{T}^\sharp$ and its corresponding extended $\mathcal K\in \mathcal{T}^{cov}$}
. Then there is a constant $C > 0$, depending only on \TODO
{the shape-regularity of 
$\mathcal{T}^{cov}$} and the polynomial order $\mathfrak{p}
$
 of $v$, such that
$\|v\|^2_{
{\mathcal{K}}
} 
\le C \|v\|^2_{
{K}
} 
$}
. 
%
%
%
Here, each component of $\nabla \bm{v}_h$ has been treated as  $v$   iteratively to each neighboring pair $\{K_{i-1}, K_i\}$ and we take the desired estimate. 
The 
{first} inequality of \eqref{ext_p} 
follows similarly 
following the same procedure 
for $q_h$\EK{.} 
%
The 
{second} inequalities of \EK{\eqref{ext_v}-\eqref{ext_p}} 
can be derived straightforwardly.
%
\end{proof}
%
%
%
%
With this preliminary result in place, we are now ready to prove discrete coercivity of $\widetilde{a}_h
$ and continuity:
\begin{lem}\label{a_coerc}\label{a_bound_revised}
For suitably large  discontinuity penalization   parameter $\sigma>0$ in the definition of the bilinear form $a_h(\cdot, \cdot)$, there exists a constant $c_a>0$, such that  
\begin{equation}\label{coercivit}
 {\green{\widetilde{a}_h(\mathbf{v}_h, \mathbf{v}_h) \ge}} c_{coer} \vertiii{\mathbf{v}_h}_{V^\sharp}^2
,
\end{equation} 
for any $\mathbf{v}_h \in V_h^\sharp$,
and
%
%
there exist constants $C_a , C_b>0$, such that 
\begin{eqnarray}
&{\green{\widetilde{a}_h(\mathbf{u}_h, \mathbf{v}_h) \leq C_a \vertiii{\mathbf{u}_h}
\cdot \vertiii{\mathbf{v}_h}
,  \quad \text{for every } \mathbf{u}_h, \mathbf{v}_h \in V_h^{\sharp},
}}
&
\label{cont3_revised}
\\
&{\red{
\widetilde{a}_h(\mathbf{u}, \mathbf{v}_h) \leq C_a \vertiii{\mathbf{u}}\cdot  \vertiii{\mathbf{v}_h}, \quad \text{for every } (\mathbf{u}, \mathbf{v}_h) \in ( [H^{k+1}(\Omega^\sharp)\cap H^1_0(\Omega^\sharp) ]^d+ V_h^{\sharp} )\times V_h^{\sharp},
}}&\label{cont1_revised}\\
&b_h(\mathbf{u}, p_h) \leq C_b \vertiii{\mathbf{u}} \cdot \vertiii{p_h},  \quad \text{for every } (\mathbf{u}, p_h) \in ( [H^{k+1}(\Omega^\sharp)\cap H^1_0(\Omega^\sharp) ]^d+ V_h^{\sharp} )\times Q_h^{\sharp},& \label{cont2_revised} \\
&b_h(\mathbf{u}_h, p) \leq C_b \vertiii{\mathbf{u}_h} \cdot \vertiii{p},  \quad \text{for every } (\mathbf{u}_h, p) \in V_h^{\sharp}\times ( [H^k(\Omega^\sharp)\cap L^2_0(\Omega^\sharp) ]+ Q_h^{\sharp} ).& \label{cont4_revised}
\end{eqnarray} 
%
%
%
\end{lem}
\begin{proof}
The proof is based on standard arguments
. In particular, for any ${\green{\lambda}}
 > 0 $, we have
\begin{align}
\widetilde{a}_h(\mathbf{v}_h, \mathbf{v}_h)
= \|\nabla \mathbf{v}_h \|^2_{\Omega^\sharp} 
+ 
 \|
{\green{\sigma^{1/2}}}
\mathbf{v}_h \|^2_{\Gamma} + \sum_{F\in \mathcal{F}^{\sharp}_{int}} \|
{\green{\sigma^{1/2}}}
[\![\mathbf{v}_h ]\!] \|^2_{\EK{F}
}
  - 2\int_{\Gamma}\mathbf{v}_h {\mathbf{\Pi_{L^2}}}(\nabla \mathbf{v}_h)\cdot \mathbf{n}_{\Gamma}\,d\gamma
  \notag 
\\
- 2\sum_{F\in \mathcal{F}^{\sharp}_{int}}\int_{\EK{F}
}\llrrparen{{\mathbf{\Pi _{L^2}}}(\nabla \mathbf{v}_h)}\cdot \mathbf{n}_F [\![\mathbf{v}_h ]\!]\,d\gamma
\notag \\
\geq  \|\nabla \mathbf{v}_h \|^2_{\Omega^\sharp} + 
 \|
{\green{\sigma^{1/2}}}
\mathbf{v}_h \|^2_{\Gamma} + \sum_{F\in \mathcal{F}^{\sharp}_{int}} \|
{\green{\sigma^{1/2}}}
[\![\mathbf{v}_h ]\!] \|^2_{\EK{F}
}
-  {\green{\lambda}}{\magenta{\sigma}}
  \|{\green{\sigma^{-1/2}}}
{\mathbf{\Pi _{L^2}}}(\nabla\mathbf{v}_h)\cdot \mathbf{n}_{\Gamma} \|^2_{\Gamma}
-{\green{\lambda}}
^{-1} 
{\magenta{\sigma^{-1}}}
 \| {\green{\sigma^{1/2}}}
\mathbf{v}_h  \|^2_{\Gamma}
\notag 
\\
- {\green{\lambda}}{\magenta{\sigma}}
\sum_{F\in \mathcal{F}^{^\sharp}_{int}} \|{\green{\sigma^{-1/2}}}
\llrrparen{\mathbf{\Pi_{L^2}}(\nabla\mathbf{v}_h)}\cdot \mathbf{n}_F \|^2_{\EK{F}
} - 
{\green{\lambda^{-1}}}{\magenta{\sigma^{-1}}}
\sum_{F\in \mathcal{F}^{\sharp}_{int}} \|{\green{\sigma^{1/2}}}
[\![\mathbf{v}_h ]\!] \|^2_{\EK{F}
}
  \notag \\
 \geq   \|\nabla \mathbf{v}_h \|^2_{ \Omega ^\sharp } +
 ({\green{1-\lambda^{-1}{\magenta{\sigma}^{-1}}}}
)
 (  \|{\green{\sigma^{1/2}}}
\mathbf{v}_h \|^2_{\Gamma} + \sum_{F\in \mathcal{F}^{\sharp}_{int}} \|{\green{\sigma^{1/2}}}
[\![\mathbf{v}_h ]\!] \|^2_{\EK{F}
} )  \notag
\\
 - {\green{\lambda}}{\magenta{\sigma}}
 (\sum_{F\in \mathcal{F}^{\sharp}_{int}} \|
{\green{\sigma^{-1/2}}}
\llrrparen{{\mathbf{\Pi _{L^2}}}({\nabla\mathbf{v}_h})}\cdot \mathbf{n}_F \|^2_{\EK{F}
} +   \|
{\green{\sigma^{-1/2}}}
{\mathbf{\Pi _{L^2}}}(\nabla\mathbf{v}_h)\cdot \mathbf{n}_{\Gamma} \|^2_{\Gamma} ). \label{coerc}
\end{align}
\normalsize
{\magenta{\EK{Young's inequality $ab\le a^2/(2\epsilon)+\epsilon b^2/2$} and inverse estimates 
\eqref{eq:Fi_to_K}, \eqref{estimate:normal_der_v}, are applied  to the latter term in (\ref{coerc}) to achieving a lower bound.}} 
In particular, note for $F \in \mathcal{F}_{int}^{\sharp}$ with $F=\partial K \cap \partial K^{'}$ that
\begin{align*}
 \|{\green{\sigma^{-1/2}}}\llrrparen{{\mathbf{\Pi _{L^2}}}(\nabla\mathbf{v}_h)}\cdot \mathbf{n}_F \|_{\EK{F}
}
\leq\frac{ ( \|{\green{\sigma^{-1/2}}}{\mathbf{\Pi _{L^2}}}(\nabla\mathbf{v}_h)\cdot \mathbf{n}_F \|_{F\subset \partial K  } +  \|{\green{\sigma^{-1/2}}}{\mathbf{\Pi _{L^2}}}(\nabla\mathbf{v}_h)\cdot \mathbf{n}_F \|_{F\subset \partial K^{'} } )}{2}
\\
\le
\frac{
{\green{\mathfrak{p}^2}} (
C_{INV}({\green{\mathfrak{p}}}, K, F_i)\frac{|F_i|}{|K|}||
{\mathbf{\Pi _{L^2}}}({\green{\sigma^{-1/2}}}\nabla\mathbf{v}_h)\cdot \mathbf{n}_F
||^2_K
+
C_{INV}({\green{\mathfrak{p}}}, K', F_i)\frac{|F_i|}{|K'|}||
{\mathbf{\Pi _{L^2}}}({\green{\sigma^{-1/2}}}\nabla\mathbf{v}_h)\cdot \mathbf{n}_F||^2_{K'}
 )}{2}
\\ 
\le \frac{1}{2}{\green{\mathfrak{p}^2}} \max_{\kappa=K,K^{'}} \{
C_{INV}({\green{\mathfrak{p}}}, \kappa, F_i)\frac{|F_i|}{|\kappa|}
 \|{\green{\sigma^{-1/2}}}{\mathbf{\Pi _{L^2}}}(\nabla{\mathbf{v}_h}) \|_{\kappa} \}
\\ 
\le 
C_{\mathfrak{p},1}\max_{{{\kappa
}} = K,K^{'}} \{ \|{{\green{\sigma^{-1/2}}}\mathbf{\Pi _{L^2}}}(\nabla{\mathbf{v}_h}) \|_{\kappa} \}
\end{align*}
and then summing over all 
 faces in the active mesh, 
\begin{equation}\label{par1}
\sum_{\substack{F\in \mathcal{F}_{int}
^{\sharp
}
\\ F=\partial K \cap \partial K'}
}
 \|{\green{\sigma^{-1/2}}}\llrrparen{\TODO{{\mathbf{\Pi _{L^2}}}}(\nabla\mathbf{v}_h)}\cdot \mathbf{n}_F \|^2_{\EK{F}
}
{\magenta{\le 
  {C}_{max}  \|\nabla \mathbf{v}_h \|_{\Omega^{\sharp}}^2}}.
\end{equation}
Likewise, using \eqref{estimate:normal_der_v}
\begin{eqnarray}\label{par2}
 \|{\green{\sigma^{-1/2}}}
\TODO{{\mathbf{\Pi _{L^2}}}}(\nabla\mathbf{v}_h)\cdot \mathbf{n}_{\Gamma} \|^2_{\Gamma} = 
\sum_{\TODO{K
}\cap \Gamma\neq \emptyset
}
 \|
{\green{\sigma^{-1/2}}}
\TODO{{\mathbf{\Pi _{L^2}}}}(\nabla\mathbf{v}_h)\cdot \mathbf{n}_{\Gamma} \|^2_{\TODO{K
}\cap \Gamma}  \nonumber
\\
\le \sum_{
K
\cap \Gamma\neq \emptyset
}
{\green{\sigma^{-1/2}}}
\sum_{i \in I_F^K}
C_{INV}({\green{\mathfrak{p}}}, K, F_i^K)\frac{{\green{\mathfrak{p}}}^2 |F_i|}{|K|}
||\mathbf{\Pi _{L^2}}(\nabla\mathbf{v}_h)\cdot \mathbf{n}_{\Gamma}
||^2_K
\nonumber\\ \le \sum_{
K
\cap \Gamma\neq \emptyset
}
{\green{\sigma^{-1/2}}}|I_F^K| \max_{i \in I_F^K} 
\{
C_{INV}({\green{\mathfrak{p}}}, K, F_i^K)
\}\frac{{\green{\mathfrak{p}}}^2 |F_i|}{|K|}
||\nabla\mathbf{v}_h
||^2_K
\nonumber\\
\le 
C _{\mathfrak{p},2}
\sum_{\TODO{\EK{K}}\cap \Gamma\neq \emptyset
} \|\nabla\mathbf{v}_h \|^2_{\EK{K}}  
{\magenta{
\le C'_{max}
   \|\nabla \mathbf{v}_h \|^2_{\Omega^{\sharp}}}}. 
\end{eqnarray}
Then, application of (\ref{ext_v}) verifies, for a suitable   choice of ${\green{\lambda}}
$, that  the terms in (\ref{par1}) and (\ref{par2}) can be  dominated by the leading two terms in (\ref{coerc}). Indeed, letting {\magenta{$C_{max}
$ and  $C'_{max}
$}} the constants in (\ref{par1}) and \eqref{estimate:normal_der_v} 
 respectively and  collecting all estimates, we conclude 
\begin{align*}
\widetilde{a}_h(\mathbf{v}_h, \mathbf{v}_h) 
\geq  (C_{u}^{-1}-\lambda{\magenta{\sigma}} (C_{max}+C'_{max}
)
 ) \|\nabla \mathbf{v}_h \|^2_{\Omega^{
\sharp
}} 
+(1-\lambda^{-1}{\magenta{\sigma}^{-1}})
\big( %
\|\sigma^{1/2}
\mathbf{v}_h
\|^2_{\Gamma}
+ \sum_{F\in \mathcal{F}_{int}^{\sharp}}
\| 
\sigma^{1/2}
[\![\mathbf{v}_h ]\!]
\|^2_{\EK{F}
}
\big).
\end{align*}
Coercivity (\ref{coercivit}) is already verified  for  $1> \lambda^{-1}{\magenta{\sigma}^{-1}}>C_u
 (C_{max}+C'_{max}
)$, or  $1< \lambda{\magenta{\sigma}}<C_u
 (C_{max}+C'_{max}
)$ which is valid for $\lambda = (1+C_u(C_{max}+C'_max))/(2\sigma)$. The  corresponding coercivity constant is $c_a=\min \{C_{u}^{-1}-\lambda{\magenta{\sigma}} (
C_{max}+C'_{max}
), 1 - \lambda^{-1}{\magenta{\sigma}^{-1}} \}$.

The proof of the continuity is standard and it is omitted for brevity.
\end{proof}
We recall also from \cite{Georgoulis17,Dong19} the following lemma and corollary that will be used next. 
\begin{lem}\label{lem:hp_projection_error_bounds}
 Let $K \in \mathcal{T}^\sharp$ satisfy Assumptions \ref{ass:basic_geometry} 
 and \ref{ass:max_card}
,
  and let $\mathcal{K} \in \mathcal{T}^{cov}$ be the
corresponding simplex with $K \subset \mathcal{K}$ as in Definition \ref{def:covering_domain}
. Suppose that $v \in L^2(\Omega^\sharp)$
is such that the extension $\mathfrak{E}v|_\mathcal{K} \in H^{l_K} (\mathcal{K})$, for some $l_K \ge 0$, and that Assumption \ref{ass:max_card} 
 is satisfied.
Then, there exists an operator $\pi_\mathfrak{p} : H^{l_K}(\mathcal{K}) \to \mathcal{P}^\mathfrak{p}(\mathcal{K})$, such that
\begin{equation}\label{ineq:proj_K_to_H_elk}
||v - \pi _\mathfrak{p}v||_{H^q(K)} \le C_1 \frac{h_K^{s_K-q}}{\mathfrak{p}^{l_K-q}}||\mathfrak{E}v||_{H^{l_K} (\mathcal{K})}, 
\end{equation}
for $0 \le q \le l_K$, and
\begin{equation}\label{ineq:proj_Fi_to_H_elk}
 ||v - \pi_\mathfrak{p} v||_{F_i} \le \mathcal{C}^\frac{1}{2}_\text{ap}(\mathfrak{p}, K, F_i)|F_i|^\frac{1}{2}
 \frac{h_K^{s_K-d/2}}{\mathfrak{p}^{l_K-1/2}}||\mathfrak{E}v||_{H^{l_K} (\mathcal{K})}, \quad l_K \ge d/2,
 \end{equation}
with
$$
\mathcal{C}_\text{ap}(\mathfrak{p}, K, F_i) := C_2 \min\{
\frac{h^d_K}{|F_i| \sup\limits_{\mathbf{v}_{0,i}\in K} \min\limits_{x\in F_i} (\mathbf{d}_i \cdot \mathbf{n})},\mathfrak{p}^{d-1}\},$$
$s_K = \min\{\mathfrak{p} + 1, l_K\}$, and $C_1$, $C_2 > 0$ constants depending only on the shape-regularity of $\mathcal{K}$, $q$, $l_K$, on $C_\text{diam}$ (from Assumption \ref{ass:max_card}
) and on the domain $\Omega^\sharp$.
\end{lem}
We note the correspondence between $\mathcal{C}_\text{INV} (\mathfrak{p}, K, F_i)$ from Lemma
\ref{lem:Fi_to_K}, 
 and $\mathcal{C}_\text{ap} (\mathfrak{p}, K, F_i)$ while $h^d_K \sim |K|$ is the typical case. The key attribute of both
expressions is that they remain bounded for degenerating $|F_i|$, allowing for the
estimates (\ref{eq:Fi_to_K}) 
and (\ref{ineq:proj_Fi_to_H_elk}) 
to remain finite as $|F_i| \to 0$. 
%
\begin{cor}\label{discrete_error_revised}
The approximation errors of the extended interpolation operators {$\bm{\pi}_\mathfrak{p}$ and $\pi_\mathfrak{p}$} \itshape for $(\mathbf{u}, p) \in \left[H^{\mathfrak{p}+1}(\mathrm{\Omega^\sharp})\right]^d \times H^\mathfrak{p}(\mathrm{\Omega^\sharp})$ satisfy
\begin{eqnarray}
\vertiii{\mathbf{u}-\bm{\pi}_\mathfrak{p} \mathbf{u}} 
&\leq &C \sum_{K\in\mathcal{T}^\sharp}\frac{h_K^{\mathfrak{p}}}{\mathfrak{p}^{\mathfrak{p}-\frac{1}{2}}}\left|\left|\mathbf{u}\right|\right|_{\mathfrak{p}+1,\mathrm{K
}},
%
\label{approx_er1_revised} 
\\
\vertiii{(\mathbf{u}-\bm{\pi}_\mathfrak{p} \mathbf{u}, p-{\pi_\mathfrak{p} p})}
&\leq &C \sum_{K\in \mathcal{T}^\sharp}\frac{h_K^{\mathfrak{p}}}{\mathfrak{p}^{\mathfrak{p}-\frac{1}{2}}} \Big(\left|\left|\mathbf{u}\right|\right|_{\mathfrak{p}+1, \mathrm{K
}} + 
\frac{1}{\mathfrak{p}^\frac{1}{2}}
\left|\left|p\right|\right|_{\mathfrak{p}, \mathrm{K
}}\Big).
%
\label{approx_er2_revised} 
\end{eqnarray}
\end{cor}
%
{
\begin{proof}
It is convenient to introduce 
the auxiliary norm 
\begin{multline*}
\vertiii{\mathbf{v}}^2_{h} = 
\sum_{\redhot{\mathcal{K}} \in {\mathcal{T}^{cov}
}}||\nabla v  ||_\redhot{\mathcal{K}}^2
 +
 \|{\green{
\sigma^{1/2}}}
\mathbf{v}\|^2_{\EK{\Gamma}}
+ {\red{\sum_{F \in {\mathcal{T}^\sharp}\cap 
\EK{\mathsf{\Gamma}}}\|{\magenta{\mathfrak{p}^{-1}}}h_F^{1/2}\mathbf{n}_{\Gamma}\cdot \nabla \mathbf{v}\|^2_{F} }}
+  \sum_{F\in \mathcal{F}^{\sharp}_{int}} \|{\green{\sigma^{1/2}
}}
[\![\mathbf{v} ]\!]\|^2_{\EK{F}
}
 \\
+
\EK{\frac{1}{2}} {\red{\sum_{K \in \EK{\mathcal{T}^{cov}}
}\|{\magenta{\mathfrak{p}^{-1}}}h_K^{1/2}\nabla \mathbf{v} |_{K}\cdot \mathbf{n}_{\partial K}\|_{\partial K}^2}},
\end{multline*}
see also \cite[p.60]{Georgoulis17}.
\EK{This norm} clearly \EK{bounds} 
 $\vertiii{\mathbf{v}}$, in the sense that
$\vertiii{\mathbf{v}-\bm{\pi}_{\mathfrak{p}}\mathbf{v}}\leq
\vertiii{\mathfrak{E}^{{\magenta{\mathfrak{p}}}+1}\mathbf{v}-\bm{\pi}_\mathfrak{p}\mathbf{v}}_h.$ 
 Hence, we may prove the estimate for $\vertiii{\cdot}_h$ instead of $\vertiii{\cdot}$. 
Setting $\mathbf{e}_{\pi}=\mathfrak{E}^{{\magenta{\mathfrak{p}}}+1}\mathbf{u}-\bm{\pi}_{\mathfrak{p}}\mathbf{u}$, we take by definition
\begin{align*}
\vertiii{\mathbf{e}_{\pi}}^2_h = 
\sum_{\EK{\mathcal{K}} \in {\mathcal{T}^{cov}
}}||\nabla \mathbf{e}_{\pi} ||_\EK{\mathcal{K}}^2
%
+ 
\|{\green{{
\sigma^{1/2}}}}
\mathbf{e}_{\pi}\|^2_{\EK{\Gamma}}
+ {\red{\sum_{F \in {\mathcal{T}^\sharp}\cap 
\EK{\mathsf{\Gamma}}
}\|{\magenta{{\mathfrak{p}}^{-1}}}h_F^{1/2}\nabla \mathbf{e}_{\pi}\cdot \mathbf{n}_{F}\|^2_{F}}}
+ \sum_{F\in \mathcal{F}^{\sharp}_{int}}\|{\green{{\sigma^{1/2}
}}
}
[\![\mathbf{e}_{\pi} ]\!]\|^2_{\EK{F}
}
\\
 \qquad\qquad\qquad + {\red{\sum_{F\in \mathcal{F}^{\sharp}_{int}}\|{\magenta{{\mathfrak{p}^{-1}}}}h_F^{1/2}\llrrparen{\nabla\mathbf{e}_{\pi}}\cdot \mathbf{n}_{F}\|^2_{\EK{F}
}}}.
\end{align*}
All the above terms 
may be estimated, using the local approximation properties (\ref{ineq:proj_K_to_H_elk})--(\ref{ineq:proj_Fi_to_H_elk}), the aforementioned inverse estimates 
and the stability  of the extension operator $\mathfrak{E}^{k+1}$. For instance, 
\begin{align*}
%
\sum_{\mathcal{K}\in{T}^{cov
}} \|\nabla \mathbf{e}_{\pi}\|_{\EK{\mathcal{K}}} 
\le
\sum_{\mathcal{K}\in{T}^{cov
}} \left\| \mathbf{e}_{\pi}\right\|_{H^1{({\EK{\mathcal{K}}})}} \stackrel{(\ref{ineq:proj_K_to_H_elk
})} 
%
\le C \sum_{\mathcal{K}\in\mathcal{T}^{cov}} \frac{h_{\EK{\mathcal{K}}}^{\mathfrak{p}}}{\mathfrak{p}^{\mathfrak{p}}}||\mathfrak{E}u||_{H^{\mathfrak{p}+1} (\mathcal{K})}
%
,
\end{align*}
after choosing $q = 1$, $l_k = \mathfrak{p}+1$ and $s^k=\min(\mathfrak{p}+1,l_k)$ in (\ref{ineq:proj_K_to_H_elk}). Similarly, we derive 
$$
\sum_{F \in {\mathcal{T}^\sharp}\cap 
\EK{\mathsf{\Gamma}}
}\|{\green{\mathfrak{p}}}
h_F^{-1/2} \mathbf{e}_{\pi}\|_{F}
\stackrel{\redhot{(\ref{ineq:proj_Fi_to_H_elk
})}} \le 
\EK{C\sum_{\mathcal{K}\in\mathcal{T}^{cov}}}
\mathfrak{p}h_{\EK{\mathcal{K}}}^{-1/2}\frac{h_{\EK{\mathcal{K}}}^{\mathfrak{p}+1/2}}{\mathfrak{p}^{\mathfrak{p}+1/2}} ||\mathfrak{E}u||_{H^{\mathfrak{p}+1} (\EK{\mathcal{K}}
)}
=
\EK{C\sum_{\mathcal{K}\in\mathcal{T}^{cov}}}
\frac{h_{\EK{\mathcal{K}}}^{\mathfrak{p}}}{\mathfrak{p}^{\mathfrak{p}-\frac{1}{2}}} ||\mathfrak{E}u||_{H^{\mathfrak{p}+1} (\EK{\mathcal{K}}
)},
$$
e.g. for $d=2$, $l_k=\mathfrak{p}+1$.
{{Let $F\subset{\partial{K}}$. As we have seen in the variational formulation, the norm $||\nabla e_p\cdot\mathbf{n}_F||_F$ can be efficiently approximated by 
$||\nabla e_p\cdot \mathbf{n}_F||_{F}\le C\frac{\mathfrak{p}}{h^{1/2}_K}||\nabla e_p||_{L^2(K)}$.
We note that inserting the $\Pi_{L^2}$ projection we are losing $\mathfrak{p}^{1/2}$ accuracy, \cite{Georgoulis2010OnTS}, although this is consistent with the half power we lose from the penalization of the method and we can prove optimal a--priori error bounds. Finally, we take
{\red
\begin{eqnarray*}
\sum_{F \in {\mathcal{T}^\sharp}\cap \EK{\mathsf{\Gamma}}
}
||{\mathfrak{p}^{-1}}h_{\EK{{F}}}^{1/2}\cdot \nabla e_p \cdot {\bf{n}}_F||_{F}
\le C \sum_{\mathcal{K}\in\mathcal{T}^{cov}} {\mathfrak{p}^{-1}}h^{1/2}_{\EK{\mathcal{K}}}\frac{\mathfrak{p}}{h_{\EK{\mathcal{K}}}^{1/2}}||\nabla e_p||_{\mathcal{K}
}
\le C \sum_{\mathcal{K}\in\mathcal{T}^{cov}}\frac{h_{\EK{\mathcal{K}}}^{\mathfrak{p}}}{\mathfrak{p}^{\mathfrak{p}}}||\mathfrak{E}u||_{H^{\mathfrak{p}+1} (\mathcal{K})}
.
\end{eqnarray*}
}}
}
Proceeding in a similar fashion, we have 
\begin{eqnarray*}
{\green{\sum_{F\in \mathcal{F}^{\sharp}_{int}} \|{\green{\sigma
}}
[\![\mathbf{e}_{\pi} ]\!]\|_{\EK{F}
}=}}\sum_{F\in \mathcal{F}^{\sharp}_{int}} \|{\green{\mathfrak{p}}}
h_{\EK{{F}}}^{-1/2}[\![\mathbf{e}_{\pi} ]\!]\|_{\EK{F}
}
{\EK{\le\EK{C\sum_{\mathcal{K}\in\mathcal{T}^{cov}}}
\frac{h_{\EK{\mathcal{K}}}^{\mathfrak{p}}}{\mathfrak{p}^{\mathfrak{p}-\frac{1}{2}}} ||\mathfrak{E}u||_{H^{\mathfrak{p}+1} (\EK{\mathcal{K}}
)}
}},
\\
{\red\sum_{F\in \mathcal{F}^{\sharp}_{int}}\|\mathfrak{p}^{-1}h^{1/2}_{\EK{F}}\llrrparen{\nabla\mathbf{e}_{\pi}}\cdot \mathbf{n}_{F}\|^2_{\EK{F}
}
\EK{\le C\EK{\sum_{\mathcal{K}\in\mathcal{T}^{cov}}}
\frac{h_{\EK{\mathcal{K}}}^{\mathfrak{p}}}{\mathfrak{p}^{\mathfrak{p}}} ||\mathfrak{E}u||_{H^{\mathfrak{p}+1} (\EK{\mathcal{K}}
)}}
,
}
\end{eqnarray*}
%
and  the proof of (\ref{approx_er2_revised}) is complete. 
The proof of the estimate (\ref{approx_er2_revised}) for the approximation error in the product space is similar, considering the  auxiliary pressure norm
\begin{equation*} 
\vertiii{p}^2_h =    \sum_{\EK{\mathcal{K}} \in {\mathcal{T}^{\EK{cov}}
}}\| p\|^2_{\EK{\mathcal{K}}}\,\,\, 
+ \sum_{F \in {\mathcal{T}^\sharp}\cap \EK{\mathsf{\Gamma}}
}\| {\green{\mathfrak{p}^{-1}}}
h_{\EK{F}}^{1/2}p\|^2_{F} +\sum_{F\in \mathcal{F}^{\sharp}_{int}} \|{\green{\mathfrak{p}^{-1}}}
h_{\EK{F}}^{1/2}[\![p ]\!]\|^2_{\EK{F}
}
\end{equation*}
and proving the assertion for ${e}_{\pi} = \vertiii{\mathfrak{E}p - {\pi}_\mathfrak{p} p}_h$. {\blue{In particular we can employ a multiplicative trace inequality:
$$||q||^{2}_{\partial K} \le C (||q||_K ||\nabla q||_K + h_K^{-1} ||q||^2_K) ,\quad q \in H^1(K),
$$ \cite[p2140]{Houston2002}, \cite[p1571]{T02},
and we conclude to
\begin{eqnarray*}
\sum_{F\in  {\mathcal{T}^\sharp}} \|{\green{\mathfrak{p}^{-1}}} h_{\EK{F}}^{1/2}{e}_{\pi}
\|^2
_{\EK{F}
}
\le
C \sum_{K\in  {\mathcal{T}^\sharp}} \big(||\mathfrak{p}^{-1} h_K^{1/2}{e}_{\pi}||_K ||\mathfrak{p}^{-1} h_K^{1/2}\nabla {e}_{\pi}||_K + h_K^{-1} ||\mathfrak{p}^{-1} h_K^{1/2}{e}_{\pi}||^2_K\big)
\\
 \stackrel{q=0, \ s_K =l_K=\mathfrak{p}
 }\le
C \sum_{\substack{\mathcal{K}\in  {\mathcal{T}^{cov}}\\
                  {K}\in{\mathcal{T}^{\sharp}}}}
\big(
%
\frac{h_K^{s_K-0}}{\mathfrak{p}^{\mathfrak{p}-0}}||
\mathfrak{p}^{-1} h_K^{1/2}
\mathfrak{E}{p}
||_{H^{ \mathfrak{p}} (\mathcal{K})}
%
%
 ||
\mathfrak{p}^{-1} h_K^{1/2}
\nabla 
{e}_{\pi}
||_K + h_K^{-1} ||
\mathfrak{p}^{-1} h_K^{1/2}
{e}_{\pi}
||^2_K\big)
\\
\le
C \sum_{\substack{\mathcal{K}\in  {\mathcal{T}^{cov}}\\
                  {K}\in{\mathcal{T}^{\sharp}}}} \big(
%
\frac{h_K^{\mathfrak{p}}}{\mathfrak{p}^{\mathfrak{p}}}||\mathfrak{p}^{-1}\redhot{h_K}^{1/2}\mathfrak{E}{p}||_{H^{ \mathfrak{p}} (\mathcal{K})}
%
%
\frac{\mathfrak{p}^2}{h_K} ||
\mathfrak{p}^{-1} h_K^{1/2}
{e}_{\pi}
||_K + h_K^{-1} ||
\mathfrak{p}^{-1} h_K^{1/2}
{e}_{\pi}
||^2_K\big)
\\
\le
C \sum_{\substack{\mathcal{K}\in  {\mathcal{T}^{cov}}\\
                  {K}\in{\mathcal{T}^{\sharp}}}}\big(
%
\frac{h_K^{\mathfrak{p}}}{\mathfrak{p}^{\mathfrak{p}}}||\mathfrak{p}^{-1} h_K^{1/2}\mathfrak{E}{p}||_{H^{ \mathfrak{p}} (\mathcal{K})}
%
%
\frac{\mathfrak{p}}{h_K^{1/2}} || 
{e}_{\pi}
||_K +  
\mathfrak{p}^{-2}||
{e}_{\pi}
||^2_K\big),
\end{eqnarray*}
and finally with standard algebra 
\begin{eqnarray*}
\sum_{F\in  {\mathcal{T}^\sharp}} \|{\green{\mathfrak{p}^{-1}}} h_{\EK{F}}^{1/2}{e}_{\pi}
\|^2
_{\EK{F}
}
\le
C \sum_{
\substack{\mathcal{K}\in  {\mathcal{T}^{cov}}\\
                  {K}\in{\mathcal{T}^{\sharp}}}
}
 \big(
%
\frac{h_K^{\mathfrak{p}}}{\mathfrak{p}^{\mathfrak{p}}}||\mathfrak{E}{p}||_{H^{ \mathfrak{p}} (\mathcal{K})}
%
%
|| 
{e}_{\pi}
||_K +  
\mathfrak{p}^{-2}||
{e}_{\pi}
||^2_K\big)
\\
\le
C \sum_{\mathcal{K}\in  {\mathcal{T}^{cov}}} \big(
%
\frac{h_K^{\mathfrak{p}}}{\mathfrak{p}^{\mathfrak{p}}}||\mathfrak{E}{p}||_{H^{ \mathfrak{p}} (\mathcal{K})}
%
%
\frac{h_K^{\mathfrak{p}}}{\mathfrak{p}^{\mathfrak{p}}}||\mathfrak{E}{p}||_{H^{ \mathfrak{p}} (\mathcal{K})}
 +  
\mathfrak{p}^{-2}\frac{h_K^{2\mathfrak{p}}}{\mathfrak{p}^{2\mathfrak{p}}}||\mathfrak{E}{p}||^2_{H^{ \mathfrak{p}} (\mathcal{K})}\big)
\\
\le
C \sum_{\substack{\mathcal{K}\in  {\mathcal{T}^{cov}} }}
\big(
%
%
\frac{h_K^{2\mathfrak{p}}}{\mathfrak{p}^{2\mathfrak{p}}}||\mathfrak{E}{p}||^2_{H^{ \mathfrak{p}} (\mathcal{K})}
 +  
\frac{h_K^{2\mathfrak{p}}}{\mathfrak{p}^{2(\mathfrak{p}+1)}}||\mathfrak{E}{p}||^2_{H^{ \mathfrak{p}} (\mathcal{K})}\big)
\\
\le
C \sum_{\mathcal{K}\in  {\mathcal{T}^{cov}}}
%
%
\frac{h_K^{2\mathfrak{p}}}{\mathfrak{p}^{2\mathfrak{p}}}||\mathfrak{E}{p}||^2_{H^{ \mathfrak{p}} (\mathcal{K})},
\end{eqnarray*}

which completes the proof.
}}
\end{proof}
}
We continue with the stability for the $b_h$ \EK{proof}.
\begin{lem}\label{stab_b}
There exists $C >0$,  such that for every $p_h \in Q_h^{\sharp}$ we have
\begin{equation}\label{b_stab}
C  \left\|p_h\right\|_{\Omega^\sharp}\leq  \sup_{\mathbf{w}_h \in V_h^{\sharp}\backslash\left\{0\right\}}\frac{b_h(\mathbf{w}_h,p_h)}{\vertiii{\mathbf{w}_h}_{V^{\EK{c}
}}}
+ 
\Big(\sum_{K\in \mathcal{T}^{
{\magenta{\sharp}}
}} \left\|\TODO{\mathfrak{p}^{-2
}}h_K\nabla p_h\right\|^2_{K}
\Big)^{1/2}+
\Big(\sum_{F\in \mathcal{F}^\sharp _{int}}\left\|
\TODO{\mathfrak{p}^{-1
}}
h_F^{1/2}
 [\![p_h ]\!]\right\|^2_{\EK{F}
 }
\Big)^{1/2}.
\end{equation}
\end{lem}
\begin{proof}
Considering a fixed $p_h \in Q_h^{\sharp}$, due to the surjectivity of the divergence operator there exists a 
$\mathbf{v}_{p_h} \in \left[H_0^1(\Omega^\sharp)\right]^d$, such that 
\begin{equation}\label{est}
\nabla \cdot \mathbf{v}_{p_h} = p_h \text{\,\, (a)} \quad \text{and} \quad C_{\Omega^\sharp}\left\|\mathbf{v}_{p_h}\right\|_{1, \Omega^\sharp}\leq \left\|p_h\right\|_{\Omega^\sharp} \text{\,\, (b)}
\end{equation}
for some constant $C_{\Omega^\sharp}>0$. 
Then,  applying integration by parts on each element and the fact that $\mathbf{v}_{p_h}$ and  $ [\![\mathbf{v}_{p_h}]\!]$  vanish on $\Gamma$ and on $F \in \mathcal{F}^{\sharp}_{int}$ respectively --since  $\mathbf{v}_{p_h} \in \left[H_0^1(\Omega^\sharp)\right]^d$ is an element of the continuous space-- implies 
\begin{align*}
 \left\|p_h\right\|_{\Omega^\sharp}^2
 = \int_{\Omega^\sharp}p_h\left(\nabla \cdot \mathbf{v}_{p_h}\right)\,d\bm{x} = -\int_{\Omega^\sharp}\mathbf{v}_{p_h}\nabla p_h\,d\bm{x}+\sum_{K \in \mathcal{T}^{
{\magenta{\sharp}}
}}\int_{\EK{\partial K}
} \left(\mathbf{v}_{p_h}\cdot \mathbf{n}_{F}\right)p_h\,d\bm{x} \\
= -\int_{\Omega^\sharp}\mathbf{v}_{p_h}\nabla p_h\,d\bm{x} +\sum_{F\in  \mathcal{F}_{int}^{\sharp}}\int_{\EK{F}
} \llrrparen{ \mathbf{v}_{p_h}}\cdot \mathbf{n}_F [\![p_h ]\!]\,d\gamma +\sum_{F\in \mathcal{F}_{int}^{\sharp}}\int_{\EK{F}
}[\![\mathbf{v}_{p_h}]\!]  \cdot \mathbf{n}_F  \llrrparen{ p_h}\,d\gamma 
\\
+\int_{\EK{\mathsf{\Gamma}}\cap K} \left(\mathbf{v}_{p_h}\cdot \mathbf{n}_{\Gamma}\right)p_h\,d\gamma\\
= -\int_{\Omega^\sharp}\mathbf{v}_{p_h}\nabla p_h\,d\bm{x} +\sum_{F\in  \mathcal{F}^{\sharp}_{int}}\int_{\EK{F}
} \llrrparen{ \mathbf{v}_{p_h}}\cdot \mathbf{n}_F [\![p_h ]\!]\,d\gamma.
\end{align*}\normalsize
\EK{Next, we introduce} 
 interpolation error $\mathbf{e}_h:=\bm{\pi}_\mathfrak{p}
 \mathbf{v}_{p_h}-\mathbf{v}_{p_h}$ for $\mathbf{v}_{p_h} \mapsto 
 \bm{\pi}_\mathfrak{p}
 \mathbf{v}_{p_h} \in V_h^{\sharp}$ in the previous expression and holds that
\begin{align}
 \left\|p_h\right\|_{\Omega^\sharp} ^2 &= \int_{\Omega^\sharp}\mathbf{e}_{h}\nabla p_h \,d\bm{x} 
-\int_{\Omega^\sharp}
\bm{\pi}_\mathfrak{p}
\mathbf{v}_{p_h}\nabla p_h\,d\bm{x}
+\sum_{F\in  \mathcal{F}_{int}^{\sharp}}\int_{\EK{F}
} \llrrparen{ \mathbf{v}_{p_h}}\cdot \mathbf{n}_F [\![p_h ]\!]\,d\gamma \notag \\
 &\stackrel{(\ref{b_alt})}{=} 
 \int_{\Omega^\sharp}\mathbf{e}_{h}\nabla p_h\,d\bm{x} - b_h(
 {\pi}_\mathfrak{p}
 \mathbf{v}_{p_h}, p_h) 
- \sum_{F\in  \mathcal{F}_{int}^{\sharp}}\int_{\EK{F}
} \llrrparen{   \mathbf{e}_{h}}\cdot \mathbf{n}_F [\![p_h ]\!]\,d\gamma.
\label{rearr}
\end{align}
For the first term, the Cauchy--Schwarz inequality, \eqref{ineq:proj_K_to_H_elk} and \eqref{est}a, yields
\begin{align}
\Big|  \int_{\Omega^\sharp}\mathbf{e}_{h}\nabla p_h\,d\bm{x}
\Big|
\leq \Big(\sum_{
\EK{K}\in \mathcal{T}^{{\magenta{\sharp}}
}}
\big\|\mathfrak{p}h_{\EK{K}}^{-1/2}\mathbf{e}_h\big\|^2_{\EK{K}}\Big)^{1/2} \Big(\sum_{\EK{K}\in \mathcal{T}^{{\magenta{\sharp}}
}} \big\| \mathfrak{p}^{-1}h_{\EK{K}}^{1/2} \nabla p_h\big\|^2_{\EK{K}}\Big)^{1/2} \notag \\
\stackrel{(\ref{ineq:proj_K_to_H_elk})}{\le} C\left\|\mathbf{v}_{p_h}\right\|_{1,\Omega^\sharp}\Big(\sum_{\EK{K}\in \mathcal{T}^{{\magenta{\sharp}}
}}\big\|\mathfrak{p}^{-1}h_\EK{K}^{1/2}\nabla p_h\big\|^2_{\EK{K}}\Big)^{1/2} 
\stackrel{(\ref{est})a}{\le} C C_{\Omega^\sharp}^{-1} \left\|p_h\right\|_{\Omega^\sharp}\Big(\sum_{\EK{K}\in \mathcal{T}^{{\magenta{\sharp}}
}}\big\|\mathfrak{p}^{-1}h_{\EK{K}}^{1/2}\nabla p_h\big\|^2_{\EK{K}}\Big)^{1/2}. \label{i1}
\end{align}
Employing 
the continuity property of the extended interpolation operator 
and (\ref{est}), 
\begin{align}
\left|  b_h(\bm{\pi}_\mathfrak{p} 
\mathbf{v}_{p_h}, p_h) 
\right|
= \frac{\left|b_h(\bm{\pi}_\mathfrak{p}
 \mathbf{v}_{p_h}, p_h)\right|}{\vertiii{\bm{\pi}_\mathfrak{p}
  \mathbf{v}_{p_h}}_{V^c
}}\vertiii{\bm{\pi}_\mathfrak{p}
\mathbf{v}_{p_h}}_{V^c
} 
\leq \Big(\sup_{\mathbf{w}_h \in V_h^{\sharp} \backslash \left\{0\right\}}\frac{b_h(\mathbf{w}_h, p_h)}{\vertiii{\mathbf{w}_h}_{V^c
}}\Big)C_{
proj
}\left\|\mathbf{v}_{p_h}\right\|_{1, \Omega^\sharp}\notag \\
 \leq \Big(\sup_{\mathbf{w}_h \in V_h^{\sharp} \backslash \left\{0\right\}}\frac{b_h(\mathbf{w}_h, p_h)}{\vertiii{\mathbf{w}_h}_{V^{\EK{c}}
 }}{\EK{\Big)}}C_{
proj
}
C_{\Omega^\sharp}^{-1}\left\|p_h\right\|_{\Omega^\sharp}.\label{i2}
\end{align}
To handle 
 the third term, we follow the steps 
 similarly as \EK{we did} for the first term 
 using \eqref{ineq:proj_Fi_to_H_elk} and \EK{we conclude to}
\begin{align}
\Big
|  
\sum_{F\in  \mathcal{F}_{int}^{\sharp}}\int_{\EK{F}
} \llrrparen{   \mathbf{e}_{h}}\cdot \mathbf{n}_F [\![p_h ]\!]\,d\gamma
\Big
|
\leq 
\Big(\sum_{F\in \mathcal{F}^{\sharp}_{int}}\big\|
{\magenta{\mathfrak{p}
}}
h_F^{-1/2}
\llrrparen{\mathbf{e}_h}\big\|^2_{\EK{F}
}
\Big)^{1/2} 
\Big(\sum_{F \in \mathcal{F}^\sharp _{int}} \big\|
\mathfrak{p}^{\magenta{-1}}
h_F^{1/2}
[\![p_h ]\!]\big\|^2_{\EK{F}
}
\Big)^{1/2} \notag \\
 \le C
\big\|\mathbf{v}_{p_h}\big\|_{1,\Omega^\sharp}
\Big(\sum_{F \in \mathcal{F}^{\sharp} _{int}} \big\|
\mathfrak{p}^{{\magenta{-1}}}
h_F^{1/2}[\![p_h ]\!]\big\|^2_{\EK{F}
}
\Big)^{1/2}  
\le C C_{\Omega^\sharp}^{-1} \left\|p_h\right\|_{\Omega^\sharp}
\Big(\sum_{F \in \mathcal{F}^{\sharp}_{int}} \big\|
\mathfrak{p}^{\magenta{-1}}
h_F^{1/2}
[\![p_h ]\!]\big\|^2_{\EK{F}}
\Big)^{1/2}. \label{i3}
\end{align}
Finally, we collect the inequalities (\ref{i1})-(\ref{i3}), and the proof is completed.
\end{proof}
 An instant conclusion  is the following. 
 %
 \begin{cor}\label{cornew}
For every $p_h \in Q_h^{\sharp}$, there exists $\mathbf{w}_h \in V_h^{\sharp}$, such that
\begin{equation}\label{infsup2aa}
b_h(\mathbf{w}_h,-p_h) \geq \left\|p_h\right\|_{\Omega^\sharp}^2-C_{\sigma
} 
\Big( 
\big(\sum_{K\in \mathcal{T}^{
{\magenta{\sharp}}
}} \big\|\mathfrak{p}^{{\magenta{-2}}
}h_K\nabla p_h\big\|^2_{K}
\big)^{1/2}+
\big(\sum_{F\in \mathcal{F}^\sharp _{int}}\big\|
\mathfrak{p}^{{\magenta{-1}}
}{h_F}^{1/2}
 [\![p_h ]\!]\big\|^2_{\EK{F}
 }
\big)^{1/2}
\Big)
\left\|p_h\right\|_{\Omega^\sharp},
\end{equation}
for suitable  $C_{\sigma
}>0$. 
 \end{cor}
 \begin{proof}
 We denote by $C_1$, $C_2$ the constants appearing in (\ref{i1}), (\ref{i3}), we rearrange (\ref{rearr}) \EK{and we derive} $
 b_h(\bm{\pi}_\mathfrak{p}
 \mathbf{v}_{p_h},-p_h) \geq \left\|p_h\right\|_{\Omega^\sharp}^2
-\left|        
 \int_{\Omega^\sharp}\mathbf{e}_{h}\nabla p_h \,d\bm{x}
\right|
-\left|
 \sum^{ }_{F\in  \mathcal{F}_{int}^{\sharp}}\int_{\EK{F}
 } \llrrparen{   \mathbf{e}_{h}}\cdot \mathbf{n}_F [\![p_h ]\!]
\,d\gamma\right|
$, hence,  the result clearly follows for $\mathbf{w}_h=\bm{\pi}_\mathfrak{p}
 \mathbf{v}_{p_h}$ with $C_{\sigma}=\max\left\{C_1, C_2\right\}$.
 \end{proof} 
 %
Next, we pass to the discrete inf--sup stability which is being proven below. 
\begin{thm}
\label{infs}
There is a constant $c_{bil}>0$, such that for all $(\mathbf{u}_h, p_h) \in V_h^{\sharp}\times Q_h^{\sharp}$, we have
\begin{equation}\label{infsup}
c_{bil}\vertiii{(\mathbf{u}_h, p_h)}_{V^\EK{c}
,Q^\EK{c}
}\leq \sup_{(\mathbf{v}_h, q_h)\in V_h^{\sharp}\times Q_h^{\sharp}}\frac{A_h(\mathbf{u}_h, p_h; \mathbf{v}_h, q_h) 
}{\vertiii{(\mathbf{v}_h, q_h)}_{V^\EK{c}
,Q^{\EK{c}
}}}.
\end{equation}
\end{thm}
\begin{proof} 
Similar 
 to \cite{AreKarKat22}
, let $(\mathbf{u}_h, p_h) \in V^\sharp_h\times Q^\sharp_h$ and note by Corollary \ref{cornew} that  there exists $\mathbf{w}_h \in V^\sharp_h$ in (\ref{infsup2aa}). 
As a matter of fact, there is no loss of generality in \EK{setting} 
 $\vertiii{\mathbf{w}_h}_{V^\EK{c}
}=\left\|p_h\right\|_{\mathrm{\Omega^\sharp}}$ and then inequality (\ref{infsup2aa}) combined with the 
Young inequality and factorizing with respect to $ \left\|p_h\right\|^2_{\mathrm{\Omega^\sharp}}$  we obtain 
\begin{align}
b_h(\mathbf{w}_h, -p_h) 
\geq \big(1-\frac{C_{
\sigma
}
\lambda
}{2}
\big)\left\|p_h\right\|^2_{\mathrm{\Omega^\sharp}}-\frac{C_{
\sigma
}}{2\lambda
}
\Big(\big(\sum_{K\in \mathcal{T}^{{\magenta{\sharp}}
}} \left\|\mathfrak{p}^{\magenta{-2}}
h_K\nabla p_h\right\|^2_{K}
\big)^{1/2}
+ 
\big(\sum_{F\in \mathcal{F}^\sharp _{int}}\big\|
\mathfrak{p}^{\magenta{-1}}
h_F^{1/2}
 [\![p_h ]\!]\big\|^2_{\EK{F}
}
\big)^{1/2}
\Big)
 ^2\notag \\
\geq \big(1-\frac{C_{
\sigma
}
\lambda
}{2}\big)\left\|p_h\right\|^2_{\mathrm{\Omega^\sharp}}
-\frac{C_{\sigma
}}{\lambda
}\sum_{K\in \mathcal{T}^{{\magenta{\sharp}}
}} \big\|\mathfrak{p}^{\magenta{-2}}
h_K\nabla p_h\big\|^2_{{\EK{K}
}}
-\frac{C_{\sigma
}}{\lambda
}\sum_{F \in \mathcal{F}^{\sharp}_{int}}\big\|\mathfrak{p}
^{\magenta{-1}}
h_F^{1/2}
 [\![  p_h ]\!]\big\|_{\EK{F}
}^2.
\label{b_est_n}
\end{align}
\EK{We focus now on  
showing} that for a sensible choice of parameters $\zeta
_1>0$ and $\zeta
_2>0$, there exists a constant $c_{bil}>0$ such that the test pair $(\mathbf{v}_h,q_h)=(\mathbf{u}_h, -p_h) +\zeta
_1(-\mathbf{w}_h,0){\red{+\zeta
_2(
{\magenta{-}}\mathfrak{p}^{-4}
\EK{\theta}
 \nabla p_h,0)}}$\EK{, for proper $\EK{\theta\in\mathbb{R}}$ that will be chosen later,} yields 
\begin{equation}\label{surr}
A_h(\mathbf{u}_h, p_h; \mathbf{v}_h,q_h)\geq c_{bil}\vertiii{(\mathbf{u}_h,p_h)}_{V^\EK{c}
,Q^\EK{c}
} \vertiii{(\mathbf{v}_h,q_h)}_{V^\EK{c}
,Q^\EK{c}
},
\end{equation}
whereby the desired outcome 
(\ref{infsup}) is then provided.

Consequently
, if we initially test with $(\mathbf{u}_h,-p_h)$ using the coercivity estimate (\ref{coercivit}) of $\widetilde{a}_h$, we derive
\begin{align}
A_h(\mathbf{u}_h, p_h;&\mathbf{u}_h, -p_h) 
=\widetilde{a}_h(\mathbf{u}_h, \mathbf{u}_h) 
\geq c_{a} \vertiii{\mathbf{u}_h}_{V^\EK{c}
}^2 
. 
\label{partial00}
\end{align}
\EK{Thus}
, we consider $(-\mathbf{w}_h,0)$ in \eqref{b_est_n} and \EK{utilize} the continuity estimate (\ref{cont3_revised}) of  $\widetilde{a}_h$ \EK{together with} 
the 
Young inequality, 
\begin{align}
A_h(\mathbf{u}_h, p_h;-\mathbf{w}_h, 0) 
= -\widetilde{a}_h(\mathbf{u}_h, \mathbf{w}_h) 
+ b_h(\mathbf{w}_h,-p_h)
 \geq -\frac{C_a}{2\lambda
}\vertiii{\mathbf{u}_h}_{V^\EK{c}
}^2
+
\big(1-\frac{C_{a}\lambda
}{2}-\frac{C_\sigma
\lambda
}{2}
\big)\left\|p_h\right\|^2_{\mathrm{\Omega^\sharp}}
\notag \\
-\frac{C_{\sigma
}
}{\lambda
}\sum_{K\in \mathcal{T}^{{\magenta{\sharp}}
}} \left\|\mathfrak{p}^{\magenta{-2}}
h_K\nabla p_h\right\|^2_{{\EK{K}
}}
-\frac{C_{\sigma
}}{
\lambda
}\sum_{F \in \mathcal{F}^\sharp_{int}}\big\|\mathfrak{p}
^{\magenta{-1}}
{h_F}^{1/2}
[\![  p_h ]\!]\big\|_{\EK{F}
}^2 \nonumber\\
\geq -C_1\vertiii{\mathbf{u}_h}_{V^\EK{c}
}^2
+C_2\left\|p_h\right\|^2_{\mathrm{\Omega^\sharp}}
-C_3\sum_{K\in \mathcal{T}^{{\magenta{\sharp}}
}} \left\|\mathfrak{p}^{\magenta{-2}}
h_K\nabla p_h\right\|^2_{{\EK{K}
}}
-C_3\sum_{F \in \mathcal{F}^\sharp_{int}}\big\|
\mathfrak{p}^{\magenta{-1}}
{h_F}^{1/2} [\![  p_h ]\!]\big\|_{\EK{F}
}^2,
\label{partial11}
\end{align}
where $C_1=\frac{C_a}{2\lambda
}$, $C_2=1-\frac{C_{a}+C_\sigma
}{2}\lambda
$ and $C_3=\frac{C_{\sigma
}}{\lambda
}$ are positive constants for 
small enough $0<\lambda
<\frac{2}{C_a+C_\sigma
}$. 
Therefore, to gain the desired control and to \EK{balance} 
the negative contribution \redhot{$\|\mathfrak{p}^{-4}\redhot{h_K^{2}}\nabla p_h\|_{\EK{K}
}^2$} in \eqref{partial11}, we test with \redhot{$({\magenta{-}}\mathfrak{p}^{-4}
\EK{\theta}
\nabla p_h,0)$} using the continuity estimate \eqref{cont3_revised} for $\widetilde{a}_h$, the Cauchy-Schwarz inequality, the inverse estimate 
\eqref{ineq:inv_est} 
and 
Young inequality. In particular,
\begin{align}
A_h(\mathbf{u}_h, p_h;
\redhot{{\magenta{-}}
\mathfrak{p}^{-4}
\EK{\theta}
\nabla p_h},0) 
= \widetilde{a}_h(\mathbf{u}_h, \redhot{{{\magenta{-}}}\mathfrak{p}^{-4}
\EK{\theta}
}
\nabla p_h) 
+ b_h(\redhot{{{\magenta{-}}}\mathfrak{p}^{-4}
\EK{\theta}
}
\nabla p_h,p_h)
\notag  \\ 
\geq - |\widetilde{a}_h(\mathbf{u}_h,
\redhot{{{\magenta{-}}}
\mathfrak{p}^{-4}
\EK{\theta}
}
\nabla p_h) 
| 
{{\magenta{-}}}
\sum_{K\in \mathcal{T}^{{{\magenta{\sharp}}}
}}\left\| \mathfrak{p}^{-2} 
\EK{\theta^{{1}/{2}}}
\nabla p_h \right\|_{\EK{K}
}^2
{{\magenta{+}}}
\sum_{F\in\mathcal{F}^\sharp_{int}}\int_{\EK{F}
}\llrrparen{\mathfrak{p}^{-4}
\EK{\theta}
\nabla p_h}\cdot \mathbf{n}_F [\![  p_h ]\!]\,d\gamma 
\notag \\
 \geq - C_a\vertiii{\mathbf{u}_h}_{V^\EK{c}
}\vertiii{\mathfrak{p}^{-4}
\EK{\theta}
\nabla p_h}_{V^\EK{c}
}  
{{\magenta{-}}}
\sum_{K\in \mathcal{T}^{
{{\magenta{\sharp}}}
}}\left\| \mathfrak{p}^{-2}
\EK{\theta^{1/2}}
\nabla p_h \right\|_{\EK{K}
}^2  \nonumber \\
 \TODO{
{{\magenta{+}}}
\widehat{C} 
\big(
\sum_{F\in\mathcal{F}^\sharp_{int}}
\|\mathfrak{p}^{-3}
\EK{\theta^{{3}/{4}}}
\nabla p_h\cdot \mathbf{n}_F
\|_{\EK{F}
}^2\big)^{1/2} 
\big(\sum_{F\in\mathcal{F}^\sharp_{int}}
\|\mathfrak{p}^{-1}
\EK{\theta^{1/4}}
[\![p_h ]\!]
\|_{\EK{F}
}^2
\big)^{1/2}
}
 \nonumber 
\end{align}
and we conclude to
\begin{align}
 A_h(\mathbf{u}_h, p_h;
\redhot{{\magenta{-}}
\mathfrak{p}^{-4}
\EK{\theta}
\nabla p_h},0)\geq -\frac{C_a}{2\lambda
_1}\vertiii{\mathbf{u}_h}_{V^\EK{c}
}^2 -\frac{C_a\lambda
_1}{2}\vertiii{\mathfrak{p}^{-4}
\EK{\theta}
\nabla p_h}_{V^\EK{c}
}^2 
{{\magenta{-}}}
\sum_{K\in \mathcal{T}^{
{{\magenta{\sharp}}}
}}\left\| \mathfrak{p}^{-2}
\EK{\theta^{1/2}}
\nabla p_h \right\|_{\EK{K}
}^2  \nonumber \\
{{\magenta{+}}}
 C 
\big(\sum_{K\in\mathcal{T}^{{{\magenta{\sharp}}}
}}
\| \TODO{{\mathfrak{p}}^{-2}
\EK{\theta^{3/4}{h^{-1/2}_K}}
}\TODO{\nabla}p_h
\|_{K}^2
\big)^{1/2} 
\big(\sum_{F\in\mathcal{F}^\sharp_{int}}
\|\mathfrak{p}^{-1}
\EK{\theta^{1/4}}
[\![p_h ]\!]
\|_{\EK{F}
}^2
\big)^{1/2} \nonumber \\
\stackrel{(\ref{aux_V}),\, \theta=h^2_K}{\geq}
 -\frac{C_a}{2\lambda
_1}\vertiii{\mathbf{u}_h}_{V^\EK{c}
}^2 
- \frac{\widetilde{C}C_a\lambda
_1}{2}
\EK{\sum_{\mathcal{K}\in  {\mathcal{T}^{cov}}}}
\left\|\mathfrak{p}^{-2}
{\EK{h_{K}}}
\nabla p_h \right\|^2_{\mathcal{K}}
%
{{\magenta{-}}}
 \sum_{K\in \mathcal{T}^{{{\magenta{\sharp}}}
}}
\left\|\mathfrak{p}^{-2} 
h_K
\nabla p_h \right\|_{\EK{K}
}^2 
\nonumber \\
{{\magenta{+}}}
\frac{C\lambda
_2}{2} 
\redhot{ \sum_{K\in \mathcal{T}^{{{\magenta{\sharp}}}}}}\left\|\TODO{\mathfrak{p}^{-2}
h_{\EK{K}} 
\nabla}p_h\right\|_{\EK{K}}
^2 
{{\magenta{+}}}
\frac{C}{2\lambda
_2}\sum_{F\in\mathcal{F}^\sharp_{int}}\left\|\mathfrak{p}^{-1}
h_F^{1/2}
[\![p_h ]\!]\right\|_{\EK{F}
}^2 \nonumber \\
 \geq
 -\frac{C_a}{2\lambda_1}\vertiii{\mathbf{u}_h}_{V^\EK{c}
}^2 
{{\magenta{-}}}
\sum_{K\in \mathcal{T}^{{{\magenta{\sharp}}}
}}\left\|\mathfrak{p}^{-2} 
h_K
\nabla p_h \right\|_{\EK{K}
}^2 
- \frac{C_p}{2}(\widetilde{C}C_a\lambda_1
{{\magenta{-}}}
C \lambda_2) \left\|\TODO{\mathfrak{p}^{-2}h_{\EK{K}} \nabla}p_h\right\|_{\mathrm{\Omega^{
{{\magenta{\sharp}}}
}}}^2 
\nonumber \\
{{\magenta{+}}}
\frac{C}{2\lambda_2}\sum_{F\in\mathcal{F}^\sharp_{int}}\left\|\mathfrak{p}^{-1}
h_F^{1/2}
[\![p_h ]\!]\right\|_{\EK{F}
}^2 \nonumber \\
\geq 
-C_4\vertiii{\mathbf{u}_h}_{V^\EK{c}
}^2 
+ 
C_5\sum_{K\in \mathcal{T}^{{{\magenta{\sharp}}}
}}\left\| \mathfrak{p}^{-2}
h_K
\nabla p_h \right\|_{\EK{K}
}^2  
{\magenta{+}}
C_6\sum_{F\in\mathcal{F}^\sharp_{int}}\left\|\mathfrak{p}^{-1}
h_F^{1/2}
[\![p_h ]\!]\right\|_{\EK{F}
}^2 \label{partial22}
\end{align}
%
%
{\green{where 
 $C_4=\frac{C_a}{2\lambda
_1}$,
 $C_5=1 - \frac{c_p}{2}(\widetilde{C}C_a\lambda_1{{\magenta{-}}}
{C}{\lambda_2})$ and $C_6=\frac{C}{2\lambda
_2}$ are positive constants for properly chosen 
$\lambda_1 > 2/(c_p \widetilde{C}C_a)>0$ and $0<\lambda_2<2/(c_pC)(c_p\widetilde{C}C_a\lambda_1/2-1)$.}}

We note that in the last inequality  $\widetilde{C}>0$, \TODO{we take $\sigma=\mathfrak{p}h_{\EK{K}}^{-1/2}$} and we derive the bound 
\begin{eqnarray}\label{aux_V}
\vertiii{\mathfrak{p}^{-4}
\EK{\theta}
\nabla p_h}^2_{V^\EK{c}
} = \left\|\mathfrak{p}^{-4}
\EK{\theta}
\nabla\nabla p_h \right\|^2_{\mathrm{\Omega^{
{\EK{cov}}
}}} 
+ 
\EK{\sum_{F \in {\mathcal{T}^\sharp\cap\mathsf{\Gamma}}}}
\left\|\mathfrak{p}^{-3}
\EK{\theta}
h^{-1/2}_{\EK{F}}\nabla p_h\right\|^2_\EK{F}
+ \sum_{F\in \mathcal{F}^\sharp_{int}} \left\|\mathfrak{p}^{-3}
\EK{\theta}
h^{-1/2}_{\EK{F}}
[\![\nabla p_h ]\!]\right\|^2_{\EK{F}
}
\nonumber\\ 
\leq 
\EK{\sum_{\mathcal{K}\in  {\mathcal{T}^{cov}}}}
\left\|\mathfrak{p}^{-2}
\EK{\theta}
{\EK{h^{-1}_{K}}}
\nabla p_h \right\|^2_{\mathcal{K}}
+ 
\EK{\sum_{F \in {\mathcal{T}^\sharp\cap\mathsf{\Gamma}}}}
\left\|\mathfrak{p}^{-3}
\EK{\theta}
h^{-1/2}_{\EK{F}}\nabla p_h\right\|^2_\EK{F}
+ \sum_{F\in \mathcal{F}^\sharp_{int}} \left\|\mathfrak{p}^{-3}
\EK{\theta}
h^{-1/2}_{\EK{F}}
[\![\nabla p_h ]\!]\right\|^2_{\EK{F}
}\nonumber
\\
\leq
 \widetilde{C}
\EK{\sum_{\mathcal{K}\in  {\mathcal{T}^{cov}}}}
\left\|\mathfrak{p}^{-2}
{\EK{h_{K}}}
\nabla p_h \right\|^2_{\mathcal{K}}
\end{eqnarray} 
which has been obtained by the trace inequalities (\ref{ineq:inv_est}
), the inverse inequality  (\ref{estimate:nabla_v}
) \EK{and finally setting $\theta=h^2_K$}. 
In particular, if we regard the norm on a facet $F\subset \partial K\in \mathcal{T}^{{\magenta{\sharp}}
}
$, 
$$
\TODO{\left\| \sigma\mathfrak{p}^{-4}h_{\EK{F}}^{2}\nabla p_h\right\|_{\EK{F}
} \leq}
\left\|\mathfrak{p}^{-3}h_{\EK{F}}^{3/2}\nabla p_h\right\|_{\EK{F}
}
 \leq \left\|\mathfrak{p}^{-3}h_{\EK{F}}^{3/2}\nabla p_h\right\|_{\partial K} 
\le C\mathfrak{p}^{-2}h_{\EK{K}}\left\|\nabla p_h\right\|_{K}
,
$$
by (\ref{ineq:inv_est}
) and   (\ref{estimate:nabla_v}
) respectively. Then, the norm corresponding to the jump on $F=K\cap K'$ satisfies
$
\left\|\mathfrak{p}^{-3}h_{\EK{F}}^{3/2}[\![\nabla p_h ]\!]\right\|_{\EK{F} 
}
\le C \mathfrak{p}^{-2} h_{\EK{K}}
\max\left\{\left\|[\nabla p_h\right\|_{K}, \left\|[\nabla p_h\right\|_{K'}\right\}
$ leading to the estimate
 $$
\sum_{F\in \mathcal{F}^\sharp_{int}} {\left\|\mathfrak{p}^{-3}h_{\EK{F}}^{3/2}[\![\nabla p_h ]\!]\right\|^2_{\EK{F}
}
\le
C \EK{\sum_{K\in {\mathcal T}^\sharp
}}} \left\| \mathfrak{p}^{-2}h_{\EK{K}} \nabla p_h\right\|_{K}^2.
$$ Proceeding \EK{in a similar way} 
for the other terms, we obtain \eqref{aux_V}.

Finally
, we collect inequalities \eqref{partial00}--\eqref{partial22} and we \EK{choose} 
$(\mathbf{v}_h,q_h)=(\mathbf{u}_h, -p_h) +\zeta_1(-\mathbf{w}_h,0)+\zeta_2(\redhot{-\mathfrak{p}^{-4}h_{\EK{K}}^2} \nabla p_h,0)$. Then, \EK{we obtain} 
 that  
\begin{align}
A_h
(\mathbf{u}_h, p_h;\mathbf{v}_h, q_h) 
\geq 
{\green{c_{a} \vertiii{\mathbf{u}_h}_{V^\EK{c}
}^2}} \nonumber \\ 
{\green{+\zeta_1 (  -C_1\vertiii{\mathbf{u}_h}_{V^\EK{c}
}^2
+C_2\left\|p_h\right\|^2_{\mathrm{\Omega^\sharp}}
-C_3\sum_{K\in \mathcal{T}^{{\magenta{\sharp}}
}} \left\|\mathfrak{p}^{\magenta{-2}}
h_K\nabla p_h\right\|^2_{{\EK{K}
}}
-C_3\sum_{F \in \mathcal{F}^\sharp_{int}}\left\|
\mathfrak{p}^{\magenta{-1}}
{h_F}^{1/2} [\![  p_h ]\!]\right\|_{\EK{F}
}^2)
}}
\nonumber \\ 
{\green{\zeta_2(
-C_4\vertiii{\mathbf{u}_h}_{V^\EK{c}
}^2 
+ 
C_5\sum_{K\in \mathcal{T}^{{{\magenta{\sharp}}}
}}\left\| \mathfrak{p}^{-2}h_K\nabla p_h \right\|_{\EK{K}
}^2  
{\magenta{+}}
C_6\sum_{F\in\mathcal{F}^\sharp_{int}}\left\|\mathfrak{p}^{-1}h_F^{1/2}[\![p_h ]\!]\right\|_{\EK{F}
}^2
) }}\nonumber \\
{\green{\geq 
(c_{a}-\zeta_1C_1-\zeta_2C_4) \vertiii{\mathbf{u}_h}_{V^\EK{c}
}^2 
+\zeta_1C_2\left\|p_h\right\|_{\mathrm{\Omega^\sharp}}^2 
+(\zeta_2 C_5-\zeta_1 C_3)\sum_{K\in \mathcal{T}^{{\magenta{\sharp}}
}} \left\|\mathfrak{p}^{-2}h_K\nabla p_h\right\|^2_{{\EK{K}
}}
}}
 \nonumber \\
{\green{+(\zeta_2 C_6-\zeta_1 C_3)\sum_{F\in\mathcal{F}^\sharp_{int}}\left\|\mathfrak{p}^{-1}h_F^{1/2}[\![p_h ]\!]\right\|_{\EK{F}
}^2 .
}}
\label{partial3}
\end{align} 
{\green{We force $\zeta_1 \le C_3^{-1}\min\{ C_5,C_6\}\zeta_2$ or $\zeta_2 \ge 
{C_3}{\min\{ C_5 C_6 \}^{-1} }\zeta_1$ and we choose  $\zeta_2 = 
{2C_3
}{\min\{ C_5 C_6\}^{-1}}\zeta_1$. We finally impose $\zeta_1 \le 
{c_a}/({C_1 + 2C_3
\min\{C_5,C_6\}^{-1}})
$}}
{\green{and we conclude that}} 
\begin{align*}
A_h&(\mathbf{u}_h, p_h;\mathbf{v}_h, q_h) 
\le
C(\vertiii{\mathbf{u}_h}_{V^\EK{c}
}^2 + \vertiii{p_h}_{Q^\EK{c}
}^2)=C\vertiii{(\mathbf{u}_h,p_h)}_{V^\EK{c}
,Q^\EK{c}
}^2, {\text{ for } C>0}. 
\end{align*}
We now note that
\begin{align*}
{\green{\vertiii{(\mathbf{u}_h-\zeta_1 \mathbf{w}_h - \zeta_2 \mathfrak{p}^{-4}h_{\EK{K}}^2\nabla p_h,-p_h)}_{V^\EK{c}
,Q^\EK{c}
}^2 
 = 
\vertiii{\mathbf{u}_h - \zeta_1 \mathbf{w}_h - \zeta_2 \mathfrak{p}^{-4}h_{\EK{K}}^2\nabla p_h}_{V^\EK{c}
}^2 +\vertiii{p_h}_{Q^\EK{c}
}^2 
}}
\\
 {\green{ \leq \vertiii{\mathbf{u}_h}_{V^\EK{c}
}^2 + \zeta_1 \vertiii{\mathbf{w}_h}_{V^\EK{c}
}^2+\zeta_2\vertiii{\mathfrak{p}^{-4}h_{\EK{K}}^2\nabla p_h}_{V^\EK{c}
}^2 +\vertiii{p_h}_{Q^\EK{c}
}^2  
}}
\\
 {\green{\leq \vertiii{\mathbf{u}_h}_{V^\EK{c}
}^2 +(\zeta_1 + 
{2CC_3
}{\min\{ C_5, C_6\}}^{-1}\zeta_1+1) \vertiii{p_h}_{Q^\EK{c}
}^2  
}}
\\
 {\green{\leq \big((\zeta_1(1 + 
{2CC_3
}{{\min}\{ C_5, C_6\}}^{-1})+1\big) \vertiii{(\mathbf{u}_h,p_h)}_{V^\EK{c}
,Q^\EK{c}
}^2,
}}
\end{align*}
and the proof of (\ref{surr}) follows, 
where $c_{bil}=\frac{\min \{{\green{c_{a}-\zeta_1(C_1+\TODO{{2C_3
}{\min\{ C_5 C_6\}}^{-1})}C_4, \zeta_1C_2 \}}}}{
 {\green{\zeta_1(1 + C {2C_3
}{\min\{ C_5, C_6\}}^{-1})+1}}
 }$.
\end{proof}
 
\section{Error estimates} \label{section5}
 \TODO{
\magentaX{Estimating the inconsistency of the bilinear form by the best approximation estimates, we have the following approximate Galerkin orthogonality.}
To obtain error estimates, in this section we will assume \EK{that} 
the exact solution $(\mathbf{u}, p) \in \left[H^1_0(\Omega^\sharp)\right]^d \times L_0^2(\Omega^\sharp)$.

\begin{lem}
\label{Galerkin orthogonality}
Let $(\mathbf{u}, p) \in \left[H^2(\Omega^\sharp)\cap H^1_0(\Omega^\sharp)\right]^d  \times \left[H^1(\Omega^\sharp)\cap L_0^2(\Omega^\sharp)\right]$ be the solution to the Stokes problem \normalfont (\ref{The Stokes Problem}) \itshape and $(\mathbf{u}_h, p_h) \in V_h^\sharp \times Q_h^\sharp$ the finite element approximation in  \normalfont  (\ref{cutdg})\EK{, with $h=\max_{\mathcal{K}\in \mathcal{T}^{cov}}{h_\mathcal{K}}
$ due to shape regularity of $\mathcal{T}^{cov}$, see also  Assumption \ref{ass:max_card}}.  
\itshape Then, 
\begin{equation}\label{galerkin}
A_h(\mathbf{u}-\mathbf{u}_h, p - p_h; \mathbf{v}_h, q_h) = \mathcal{O}(
{h^{\mathfrak{p}}}
/
{\mathfrak{p}^{\mathfrak{p}-\frac{1}{2}}})
 \quad \text{for every} \quad  (\mathbf{v}_h, q_h) \in V_h ^\sharp\times Q_h^\sharp.
  \end{equation} 
\end{lem}
%
}
Next, we prove the a priori error estimate and optimal convergence rates for the velocity and pressure.
\begin{thm}
\label{order_revised}
Let $(\mathbf{u}, p) \in \left[H^{\mathfrak{p}+1}(\Omega^\sharp)\cap H^1_0(\Omega^\sharp)\right]^d \times \left[H^\mathfrak{p}(\Omega^\sharp)\cap L_0^2(\Omega^\sharp)\right]$ be the solution to the Stokes problem \normalfont \eqref{The Stokes Problem} \itshape and $(\mathbf{u}_h, p_h) \in V_h^\sharp \times Q_h^\sharp$ the finite element approximation \EK{in agreement with} 
\normalfont \eqref{cutdg}. \itshape Then, there exists a constant $C>0$, such that
\begin{equation}\label{apriori_revised}
\vertiii{(\mathbf{u}-\mathbf{u}_h, p-p_h)}
\leq C \EK{\sum_{K\in\mathcal{T}^\sharp}\frac{h_K^{\mathfrak{p}}}{\mathfrak{p}^{\mathfrak{p}-\frac{1}{2}}}
\Big(
||\mathbf{u}||_{\mathfrak{p}+1, K
} + 
\frac{1}{\mathfrak{p}^\frac{1}{2}}||p||_{\mathfrak{p},K
}\Big)
}
%
.
\end{equation}
\end{thm}
\begin{proof}
We start by  rearranging the  $(\mathbf{u}-\mathbf{u}_h, p-p_h)$ error adding and subtracting appropriate terms, and we conclude to its discrete  and projection error components:
$$
\vertiii{(\mathbf{u}-\mathbf{u}_h, p-p_h)}\leq \vertiii{(\mathbf{u}-\bm{\pi}_\mathfrak{p}
\mathbf{u}, p - {\pi}_\mathfrak{p}
p)}+\vertiii{(\bm{\pi}_\mathfrak{p}
\mathbf{u}-\mathbf{u}_h, {\pi}_\mathfrak{p}
p-p_h)}.
$$
The first term optimal estimate is provided by (\ref{approx_er1_revised}) and (\ref{approx_er2_revised}) as it is proven in Corollary \ref{discrete_error_revised}. The second term first is  bounded by
$$
\vertiii{(\bm{\pi}_\mathfrak{p}
\mathbf{u}-\mathbf{u}_h, {\pi}_\mathfrak{p}
p-p_h)}\leq C\vertiii{(\bm{\pi}_\mathfrak{p}
\mathbf{u}-\mathbf{u}_h, {\pi}_\mathfrak{p}
p-p_h)}_{V^\EK{c}
,Q^\EK{c}
},
 $$
\EK{according to} (\ref{norm_est1}). Due to 
 (\ref{galerkin}) from Lemma \ref{Galerkin orthogonality} and Theorem, \ref{infs} there exists a pair $(\mathbf{v}_h, q_h) \in V_h^\sharp\times Q_h^\sharp$ with $\left\|(\mathbf{v}_h, q_h)\right\|_{V^\EK{c}
,Q^\EK{c}
}=1$, such that 
 \begin{align*}
  c_{bil}\vertiii{\left(\bm{\pi}_\mathfrak{p}
  \mathbf{u}-\mathbf{u}_h, {\pi}_\mathfrak{p}
  p - p_h\right)}_{V^\EK{c}
,Q^\EK{c}
}  \leq A_h(\bm{\pi}_\mathfrak{p}
\mathbf{u}-\mathbf{u}_h, {\pi}_\mathfrak{p}
p - p_h; \mathbf{v}_h, q_h)  
  = A_h(\bm{\pi}_\mathfrak{p}
  \mathbf{u}-\mathbf{u}, {\pi}_\mathfrak{p}
  p - p; \mathbf{v}_h, q_h).
 \end{align*}
Hence, we 
 use the definition of the corresponding form $A_h$ to take
 \begin{align}
A_h(\bm{\pi}_\mathfrak{p}
\mathbf{u}-\mathbf{u}, {\pi}_\mathfrak{p}
 p - p; \mathbf{v}_h, q_h)=\widetilde{a}_h(\bm{\pi}_\mathfrak{p}
 \mathbf{u}-\mathbf{u}, \mathbf{v}_h) +b_h(\mathbf{v}_h, {\pi}_\mathfrak{p}
  p - p) + b_h(\bm{\pi}_\mathfrak{p}
  \mathbf{u}-\mathbf{u},  q_h). 
\label{last}
\end{align}
\EK{Next}, recalling that the pair $\left(\mathbf{v}_h,q_h\right)$ has unit $\vertiii{\cdot}_{V^\EK{c}
,Q^\EK{c}
}$ norm, we derive
$$
A_h(\bm{\pi}_\mathfrak{p}
\mathbf{u}-\mathbf{u}, {\pi}_\mathfrak{p}
 p - p; \mathbf{v}_h, q_h)
 \le C_a \vertiii{\bm{\pi}_\mathfrak{p}
 \mathbf{u}-\mathbf{u}}\cdot \vertiii{v_h} + C_b \vertiii{v_h}\cdot \vertiii{{\pi}_\mathfrak{p}
 p - p_h} + C_b \vertiii{\bm{\pi}_\mathfrak{p}
 \mathbf{u}-\mathbf{u}}\cdot \vertiii{q_h},
$$
and after employing the continuity of $\widetilde{a}_h$ and $b_h$ in (\ref{cont1_revised})--(\ref{cont4_revised}) and Corollary \ref{discrete_error_revised},  also bounds  the remaining terms. At this end, an estimate for (\ref{last}) follows as
\begin{eqnarray*}
A_h(\bm{\pi}_\mathfrak{p}
\mathbf{u}-\mathbf{u}, {\pi}_\mathfrak{p}
 p - p; \mathbf{v}_h, q_h)
\EK{\le  C\sum_{K\in\mathcal{T}^\sharp}\frac{h_K^{\mathfrak{p}}}{\mathfrak{p}^{\mathfrak{p}-\frac{1}{2}}} \left\{\max{\{\vertiii{v_h}\cdot\vertiii{q_h}}\}\left(3 ||\mathbf{u}||_{\mathfrak{p}+1, K
} + \frac{1}{\mathfrak{p}^\frac{1}{2}}||p||_{\mathfrak{p}, K
} \right)\right\}}
\\
\le  \EK{C\sum_{K\in{\mathcal{T}^\sharp}}\frac{h_K^{\mathfrak{p}}}{\mathfrak{p}^{\mathfrak{p}-\frac{1}{2}}}\Big(||\mathbf{u}||_{\mathfrak{p}+1, K
} + \frac{1}{\mathfrak{p}^{\frac{1}{2}}}||p||_{\mathfrak{p},K
}\Big),}
\end{eqnarray*}
and the estimate \eqref{apriori_revised} follows.
{\blueX{
It is also necessary to 
check that the extended bilinear form $\tilde{a}_h (\cdot, \cdot)$ and the related inconsistency error bound does not affect the aforementioned  optimal error bound, see also \cite[Section 4.2]{Georgoulis17}.  This can easily be verified by employing again the tools of Corollary \ref{discrete_error_revised}. 
 In particular, we  consider the residual $ R_h (\bm{v_h} ) := \tilde{a}_h (\bm{u}, \bm{v_h} ) - {a_h(}\bm{u}, \bm{v_h} ) $, we substitute the bilinear forms with their components as in Subsection \ref{22}, we recall that $\bm{u}$ belongs in $V^\sharp$, e.g. $[\![ {\mathbf{u}} ]\!] $ is zero, on the boundary $\bm{u}|_\Gamma = 0$, and the residual becomes  
$$
-\sum_{F \in \mathcal{F}^\sharp_{int} \cup {\mathsf{\Gamma}}} \int_{F} \left( \llrrparen{{{\mathbf{\Pi_{L^2}}(}}\nabla \mathbf{u}){ - \nabla \mathbf{u}}
}\cdot \bf{n_F} [\![ {\mathbf{v}}_h  ]\!] \right )\,d\gamma 
-
 \int_{\Gamma}\left (\mathbf{v}_h {{\mathbf{\Pi_{L^2}}(}}\nabla \mathbf{u}) { - \nabla \mathbf{u}}\big )  \cdot \mathbf{n}_{\Gamma} \right ) \,d\gamma  ,
$$
or
$$
 \left | R_h (\bm{v_h} )\right | \le \sum_{{F \in \mathcal{F}^\sharp_{int} \cup {\mathsf{\Gamma}}}}
\int_{F} \left| \llrrparen{{{\mathbf{\Pi_{L^2}}(}}\nabla \mathbf{u})
{ - \nabla \mathbf{u}}}\cdot \bf{n_F} [\![ {\mathbf{v}}_h  ]\!] \right |\,d\gamma + \int_{\Gamma}\left |\mathbf{v}_h {{\mathbf{\Pi_{L^2}}(}}\nabla \mathbf{u}){ - \nabla \mathbf{u}}\big ) \cdot \mathbf{n}_{\Gamma} \right | \,d\gamma . 
$$
%
Next, we add and subtract the term $\bm{\pi}_\mathfrak{p}  \nabla \mathbf{u}$,  and after  algebraic calculations we derive that 
\begin{align*}
\left | R_h (\bm{v_h} )\right | &\le \sum_{F \in \mathcal{F}^\sharp_{int} \cup {\mathsf{\Gamma}}}\int_{F}\left |\left(\llrrparen{\nabla \mathbf{u}- \bm{\pi}_\mathfrak{p}  \nabla \mathbf{u}}\cdot \bf{n_F} [\![ {\mathbf{v}}_h  ]\!]\right |
 + \left|\llrrparen{{\mathbf{\Pi_{L^2}}(}\nabla {\mathbf{u}}) - {{\bm{\pi}}_\mathfrak{p}}  \nabla \mathbf{u}}\cdot \bf{n_F} [\![ {\mathbf{v}}_h  ]\!]\right| \right)\,d\gamma \\
&\quad  +\int_{\Gamma}\mathbf{v}_h\big (\left| \nabla \mathbf{u}- \bm{\pi}_\mathfrak{p}  \nabla \mathbf{u} \right |
+\left|{\mathbf{\Pi_{L^2}}(}\nabla \mathbf{u}) - \bm{\pi}_\mathfrak{p}  \nabla \mathbf{u}\right|\big )
 \cdot \mathbf{n}_{\Gamma}\,d\gamma 
 \\
 &\le   \sum_{F \in \mathcal{F}^\sharp_{int} \cup {\mathsf{\Gamma}}}\int_{F}\left |\left(\llrrparen{\nabla \mathbf{u} - \bm{\pi}_\mathfrak{p}  \nabla \mathbf{u}}\cdot \bf{n_F} [\![ {\mathbf{v}}_h  ]\!]\right |
 + \left|\llrrparen{{\mathbf{\Pi_{L^2}}(}\nabla {\mathbf{u}} - {{\bm{\pi}}_\mathfrak{p}}  \nabla \mathbf{u})
}\cdot \bf{n_F} [\![ {\mathbf{v}}_h  ]\!]\right| \right)\,d\gamma \\
&\quad 
+ \int_{\Gamma}\mathbf{v}_h  
\big (
\left| \nabla \mathbf{u}
- \bm{\pi}_\mathfrak{p}  \nabla \mathbf{u} 
\right |
+\left|{\mathbf{\Pi_{L^2}}       (
}\nabla \mathbf{u}
 - \bm{\pi}_\mathfrak{p}  \nabla \mathbf{u}       )
\right|
\big )
 \cdot \mathbf{n}_{\Gamma}\,d\gamma .
 \end{align*}
  Finally, for the last inequality terms we use the fact that $\mathbf{\Pi_{L^2}}{{\bm{\pi}}_\mathfrak{p}}\mathbf{v}$ coincides with ${{\bm{\pi}}_\mathfrak{p}}\mathbf{v}$,  {we follow the same procedure as in the beginning of the present proof}, the tools as demonstrated  in Collorary \ref{discrete_error_revised}, that  the projector $\mathbf{\Pi_{L^2}}$  has the stability property  
 $||\mathbf{\Pi_{L^2}}{\mathbf{v}}||_{L^2(K)}\le ||{\mathbf{v}}||_{L^2(K)}$,  
and we verify that the related inconsistency error bound does not affect 
 the aforementioned  optimal error bound. 
}}
\end{proof}
\EK{
\begin{rem}
We note that if one makes stronger assumptions as mesh quasi-uniformity: 
$h = \max_{K\in\mathcal{T}^\sharp} h_K,$  
then the estimate of Theorem \ref{order_revised} \EK{is transformed to} 
$
\vertiii{(\mathbf{u}-\mathbf{u}_h, p-p_h)}
\EK{\leq C }\frac{h^{\mathfrak{p}}}{\mathfrak{p}^{\mathfrak{p}-\frac{1}{2}}}
\Big(||\mathbf{u}||_{\mathfrak{p}+1, \Omega^\sharp} + 
\frac{1}{\mathfrak{p}^\frac{1}{2}}||p||_{\mathfrak{p},\Omega^\sharp}\Big),
$
which verifies optimal convergence in $h$ and suboptimal in $p$ by $p^{1/2}$.
\end{rem}
}
\section{Numerical Experiments/Convergence tests} \label{section6}
 %
\subsection{1st experiment: Rectangular domain $\Omega^\sharp$}
We consider a two--dimensional test case of (\ref{The Stokes Problem}) in the unit square ${\Omega^\sharp}=\left[0,1\right]^2$ with 
 exact solution
$$
\mathbf{u}\left(x,y\right)=\left(u\left(x,y\right), -u\left(y,x\right)\right), \quad p\left(x,y\right)=\sin\left(2\pi x\right)\cos\left(2\pi y\right),
$$
where $u(x,y)=\left(\cos\left(2\pi x\right)-1\right)\sin\left(2\pi y\right)$. 
%
%
%
%
%
%

Note that the mean value of $p\left(x,y\right)$ over \EK{$\Omega^\sharp$} vanishes by construction,
thus ensuring that the problem (\ref{The Stokes Problem}) is uniquely solvable. As in subsection \ref{22}, in the \redhot{spirit of arbitrarily shaped discontinuous Galerkin method on the boundary  
approach}, we consider the original domain {$\Omega^\sharp$} 
{
as seen in Figure \ref{geometry}}\greenX{'s first geometry}.  A level set description of the geometry is possible through the function
\begin{equation}\label{lset}
\phi\left(x,y\right)=\left|x-0.5\right|+\left|y-0.5\right|+\left|\left|x-0.5\right|-\left|y-0.5\right|\right|-1<0.
\end{equation}
To investigate the error convergence behavior of the discretization \eqref{cutdg}, we consider a sequence of successively refined tessellations $\{\mathcal{T}^\sharp_{h_
i}\}_{i>0}$ of ${\Omega^\sharp}$ with
mesh parameters $h_i=2^{-i-2}$, for $i=0, \ldots, 7$. 
{
As it is shown in \greenX{the rectangular geometry case of} Figure \ref{mesh_geometry} 
exploiting the  level-set function information \EK{on} $\partial\Omega^\sharp$, elements containing straight faces approximating the 
 \greenX{polytopal} boundary are marked. Finally, all marked elements are treated \greenX{as 
classical} triangular elements 
 \greenX{in the interior, with polytopal elements only on the boundary} 
described by the domain \TODO
{level-set} function, thus capturing the domain exactly.}
To allow for several polynomial degrees $\mathfrak{p}_i$, the symmetric interior penalty parameter in \eqref{a_alt} scales as $\sigma = \mathfrak{p}_i^2/h_i$. 
A sparse 
solver has been used to solve the arising linear systems. 

As expected from the theoretical error estimate stated in Theorem \ref{order_revised}, optimal $ \mathfrak{p}$-th order convergence rates with respect to the $H^{1}$--norm of the velocity error and the $L^{2}$--norm of the pressure error are in agreement with 
the error results as they are visualized 
 in Table \ref{table11} \magentaX{($ \mathfrak{p}_1/\mathfrak{p}_{0}$) and Table \ref{table22} ($ \mathfrak{p}_2/\mathfrak{p}_{1}$). Although,} as seen in Table \ref{table33} {\magentaX{($ \mathfrak{p}_3/\mathfrak{p}_{2}$) the convergence rates are optimal for the velocity and} suboptimal for pressure. {The superiority of the higher order framework is obvious 
and show that 
the method is efficient}}. Numerical experiments \EK{verify} these facts
, indeed, for larger $\mathfrak{p}$, much smaller errors are attained in progressively coarser meshes, for $ \mathfrak{p}_i/\mathfrak{p}_{i-1}$ Taylor--Hood velocity/pressure pairs
order, and the inf--sup stability is guaranteed. 
The aforementioned results are visualized in Figure \ref{sensitivity_V2_P3}.
  
\begin{table}[htbp]
\centering  
{\begin{tabular}{|c|c|c||c|c|} 
\toprule
Discretization   & \multicolumn{4}{c|}{Errors and convergence rates: $\mathfrak{p}_1 / \mathfrak{p}_0$ polynomials case}    \\ 
\midrule 
$h_{\max}$ & $\left\|\mathbf{u}-\mathbf{u}_h\right\|_{1,\Omega^\sharp}$ & Conv. rate 
           & $\left\|p-p_h\right\|_{\Omega^\sharp}$ &Conv. rate
           \\
\midrule 
$0.25 (=2^{-2})$ & 1.469461 &  --     & 0.8756151 &       --         \\
$0.125 (=2^{-3})$ & 0.871741 & 0.7533148 & 0.4336638 & 1.0137197    \\
$0.0625 (=2^{-4})$ & 0.526300 & 0.7280155 & 0.3289021 & 0.3989189     \\
$0.03125 (=2^{-5})$ & 0.290011 & 0.8597764 & 0.1901614 & 0.7904334     \\
$0.015625 (=2^{-6})$ & 0.151577 & 0.9360504 & 0.1006410 & 0.9180059     \\
$0.0078125 (=2^{-7})$ & 0.077476 & 0.9682347 & 0.0515098 & 0.9662987     \\
$0.00390625 (=2^{-8})$ & 0.039269 & 0.9803399 & 0.0261353 & 0.9788479      \\
\midrule
Mean: &    --    & {0.8709553} &     --    & {0.8443707} 
 \\
\bottomrule
\end{tabular}}
\caption{\greenX{Rectangular geometry:} Errors and rates of convergence with respect to $H^1$-norm for the velocity and $L^2$-norm for the pressure, using $\mathfrak{p}_1 / \mathfrak{p}_0$ finite elements
.}\label{table11}
\end{table}
\begin{table}[htbp]
\centering  
{\begin{tabular}{|c|c|c||c|c|} 
\toprule
  Discretization  & \multicolumn{4}{c|}{Errors and convergence rates: $\mathfrak{p}_2 / \mathfrak{p}_1$ polynomials case}  \\ 
\midrule
$h_{\max}$ & $\left\|\mathbf{u}-\mathbf{u}_h\right\|_{1,\Omega^\sharp}$ & Conv. rate   & $\left\|p-p_h\right\|_{\Omega^\sharp}$ & Conv. rate   \\
\midrule 
$0.25 (=2^{-2})$ & 0.2372591 &   --    & 0.1131108 &        --     \\
$0.125 (=2^{-3})$ & 0.1129321 & 1.0710074 & 0.0512201 & 1.1429531  \\
$0.0625 (=2^{-4})$ & 0.0337229 & 1.7436526 & 0.0165609 & 1.6289302  \\
$0.03125 (=2^{-5})$ & 0.0088867 & 1.9240082 & 0.0046825 & 1.8224281  \\
$0.015625 (=2^{-6})$ & 0.0023107 & 1.9433135 & 0.0012434 & 1.9129403  \\
$0.0078125 (=2^{-7})$ & 0.0005816 & 1.9900301 & 0.0003073 & 2.0165782  \\
\midrule
Mean: &  --   & {1.7344024} &    --     & {1.7047660}
 \\
\bottomrule
\end{tabular}}
\caption{\greenX{Rectangular geometry:} Errors and rates of convergence with respect to $H^1$-norm for the velocity and $L^2$-norm for the pressure, using  $\mathfrak{p}_2 / \mathfrak{p}_1$ and finite elements
.}\label{table22}
\end{table}
\begin{table}[htbp]
\centering  
{
\begin{tabular}{|c|c|c||c|c|} 
\toprule
  Discretization  &  \multicolumn{4}{c|}{Errors and convergence rates: $\mathfrak{p}_3 /\mathfrak{p}_2$ polynomials case} \\ 
\midrule 
$h_{\max}$ &    $\left\|\mathbf{u}-\mathbf{u}_h\right\|_{1,\Omega^\sharp}$ & Conv. rate
           & $\left\|p-p_h\right\|_{\Omega^\sharp}$ & Conv. rate \\
\midrule 
$0.25 (=2^{-2})$           & 0.0466716 &   --             & 0.0184400 &    --   \\
$0.125 (=2^{-3})$           & 0.0076981 & 2.5999546  & 0.0069364 & 1.4105717 \\
$0.0625 (=2^{-4})$           & 0.0012099 & 2.6696092  & 0.0017273 & 2.0056609\\
$0.03125 (=2^{-5})$           & 0.0001644 & 2.8789030  & 0.0009789 & {0.8192784} \\
$0.015625 (=2^{-6})$           & 2.5518e-05 & 2.6883086 & 0.0003071 & 1.6720465 \\
\midrule
Mean:       &          --       & {2.7091938} &   --      & {1.4768894}
 \\
\bottomrule
\end{tabular}}
\caption{\greenX{Rectangular geometry:} Errors and rates of convergence with respect to $H^1$-norm for the velocity and $L^2$-norm for the pressure, using  $\mathfrak{p}_3 /\mathfrak{p}_2$ finite elements
.}
\label{table33}
\end{table}
\begin{table}[htbp]
\centering  
{\begin{tabular}{|c|c|c||c|c|} 
\toprule
  Discretization  &  \multicolumn{4}{c|}{Errors and convergence rates: $\mathfrak{p}_3 / \mathfrak{p}_2$ polynomials case} \\ 
\midrule 
$h_{\max}$ &    $\left\|\mathbf{u}-\mathbf{u}_h\right\|_{1,\Omega^\sharp}$ & Conv. rate
           & $\left\|p-p_h\right\|_{\Omega^\sharp}$ & Conv. rate \\
\midrule 
$0.25 (=2^{-2})$           & 0.3386166 &   --            & 0.7184173 &    --   \\
$0.125 (=2^{-3})$         & 0.0247465 & 3.7743523  &0.0590158  & 3.60564722 \\
$0.0625 (=2^{-4})$       & 0.0039329 & 2.6535468  & 0.0096287  & 2.61568216\\
$0.03125 (=2^{-5})$     & 0.0006434 &  2.6117326 & 0.0017089 &  2.4942619 \\
$0.015625 (=2^{-6})$   & 0.0002045  & 1.6536038 & 0.0003671 &  2.2187371\\
\midrule
Mean:       &          --       & {2.67330889 
} &   --      & {2.73358212
}
 \\
\bottomrule
\end{tabular}}
\caption{\greenX{Rectangular geometry:} Errors and rates of convergence with respect to $H^1$-norm for the velocity and $L^2$-norm for the pressure, using  $\mathfrak{p}_3 / \mathfrak{p}_2$ finite elements with higher order ghost boundary elements area stabilization in the context of \cite{AreKarKat22}.}
\label{table44}
\end{table}
%

\magentaX{Moreover in} Table \ref{table44}, and in the context of \cite{AreKarKat22}, a sequence of $ \mathfrak{p}_3/\mathfrak{p}_{2}$ approximations for the 
velocity and pressure solution in progressively finer 
 \greenX{ meshes
  is} illustrated, showcasing the convergence of the method and
the errors with respect to \magentaX{several mesh sizes}. 
Further investigation, revealed that contrasting the results between Table \ref{table33} and Table \ref{table44} related to  the $ \mathfrak{p}_3/\mathfrak{p}_{2}$ case, i.e. comparing the suggested in the present work framework without 
 any ghost type of stabilization on the boundary zone area, against to the  
 full ghost penalty stabilization, showed that: the suggested method, yields better error behavior and better convergence rates \greenX{ only for the velocity, although suboptimal for the pressure even if the errors are better, \magentaX{a phenomenon that will be discussed later and it is probably caused by the quadrature on the polytopal boundary elements we employed.}
} 
Therefore, 
we \greenX{may avoid} integration errors related to the normal derivative up to the third-order jump terms \greenX{used in  the ghost penalty stabilization approach}, \cite{MLLR12,MLLR14
}, 
which lead to a significant improvement for the velocity and pressure errors, as well as, crucial computational cost savings. 
We again underline that  the evaluation of  boundary elements stabilization is avoided, {namely}, 
computations of gradually \redhot{higher order normal derivative jumps} --depended on  the finite element order-- seem unnecessary with the suggested approach.  
%
 We also re-note that the pressure convergence rate in Table \ref{table33} is suboptimal although the error is much smaller e.g. than the case where full boundary area elements stabilization is applied for $\mathfrak{p}_3 / \mathfrak{p}_2$ and as it is {reported in Table \ref{table44}}.
\begin{figure}
\centering
\begin{tabular}{cc}
\subfigure[ $\|\mathbf{u}-\mathbf{u}_h\|_{1,\Omega^\sharp}$ error, conv.rates, and mean conv. rates]{
\includegraphics[scale=0.50]{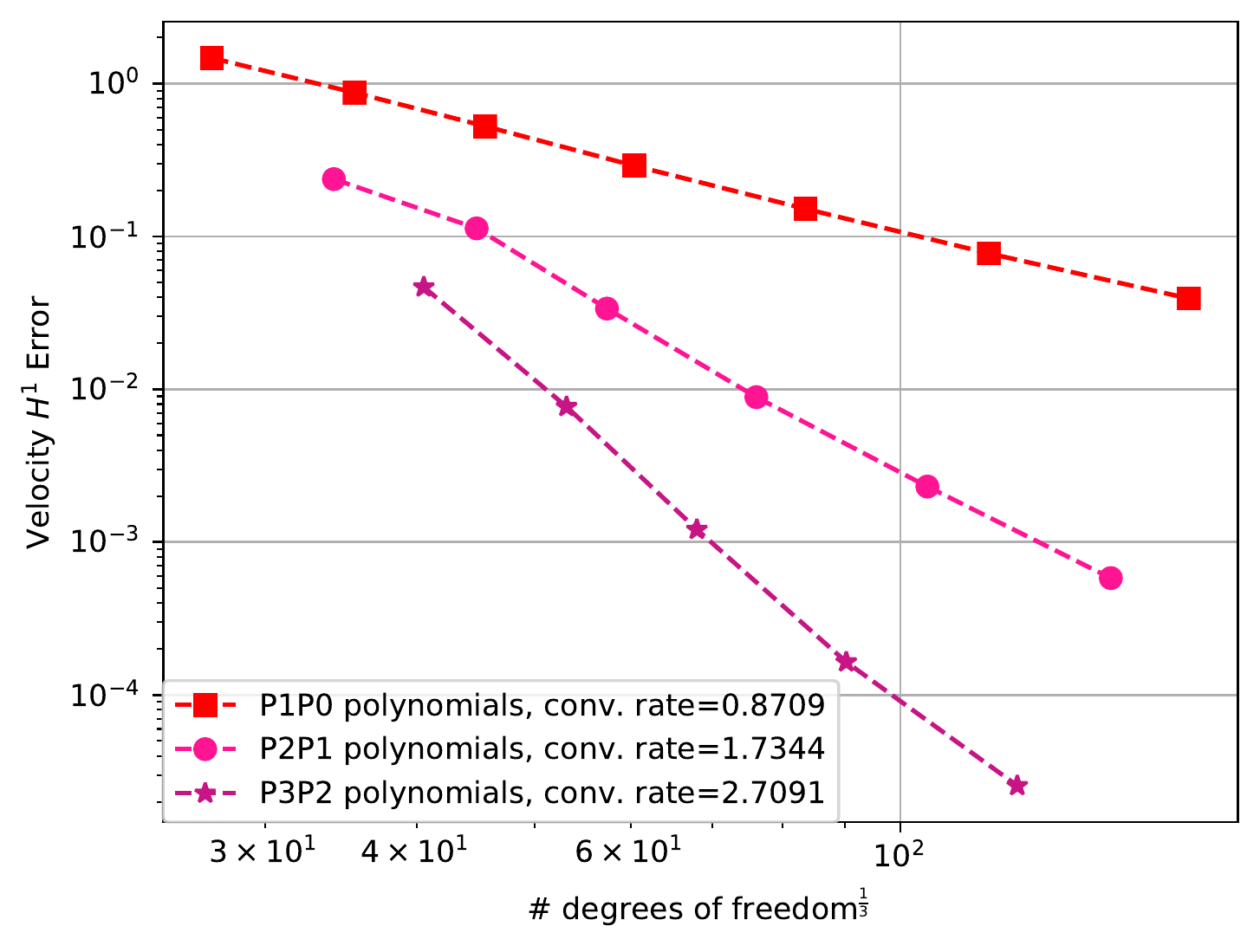}
}
\subfigure[$\|p-p_h\|_{\Omega^\sharp}$  error, conv.rates, and mean conv. rates]{
\includegraphics[scale=0.50]{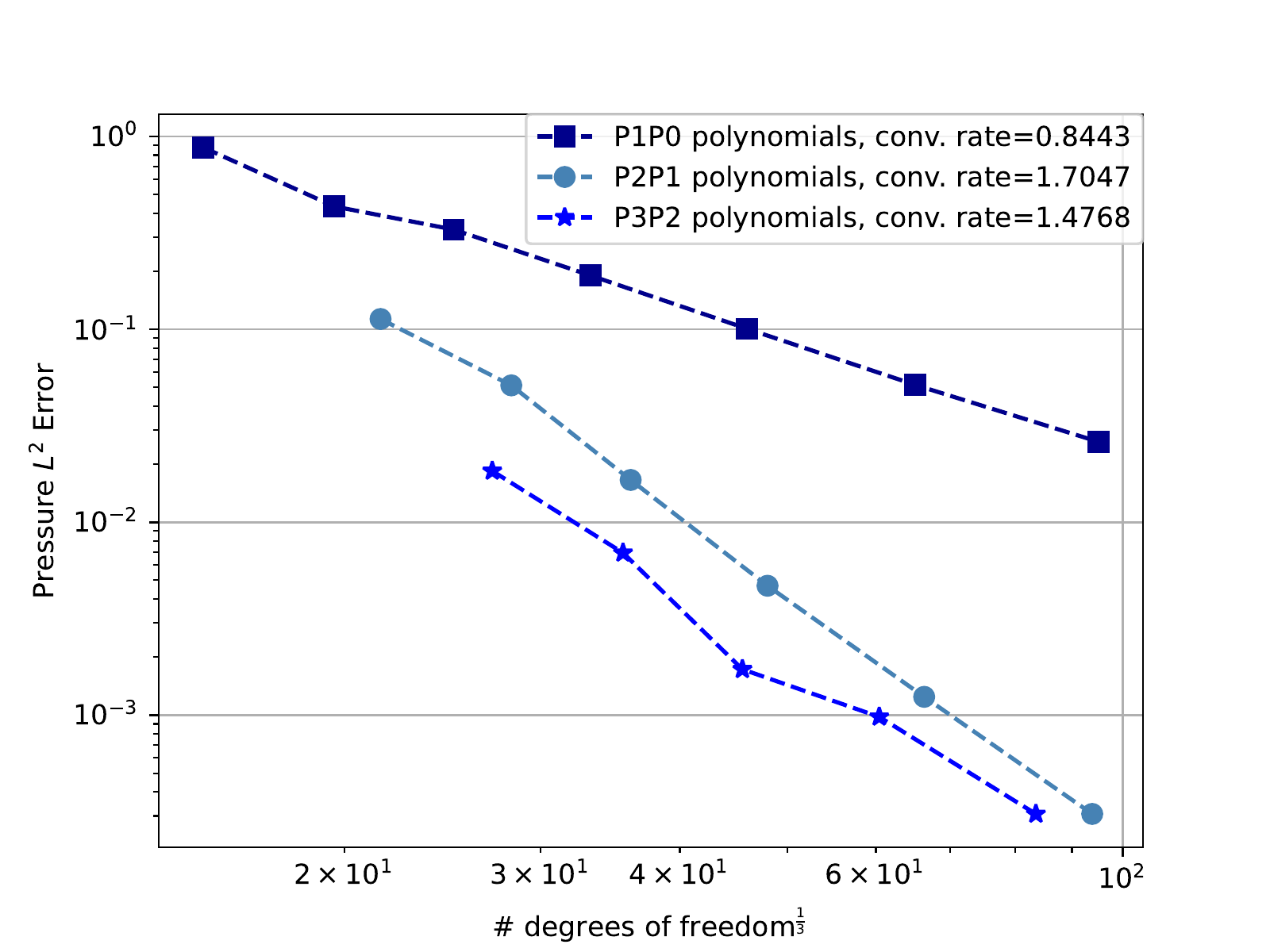}
}
\end{tabular}
\caption{\greenX{Rectangular geometry:} Visualization of the $H^1$-norm velocity errors and $L^2$-norm pressure errors with respect to the discretization size.}\label{sensitivity_V2_P3}
\end{figure}
%
Finally, we report that in practice, the condition number for arbitrarily shaped boundary elements discontinuous Galerkin methods to solve Stokes
problem is typically very large. The employment of efficient multigrid preconditioners for such cases, is left as a further challenge, see also the works \cite{AK21}, \cite{Dong19}. \blueX{We also clarify that for the {polytopal} 
 boundary elements, the current implementation performs quadrature by a sufficiently fine sub-triangulation approximating {properly} the curved element. {Although, this first implementation of ours could be the issue of noticing slightly lower convergence rate than the predicted and specially on highest order polynomial case, and even worse for the {pressure component.}} 
It is noted  that in this case the sub-triangulation is only used to generate the quadrature rules.} 
Of course, this is not the only possibility. For instance, domain-exact quadrature algorithms for many curved domains exist,  for such algorithms see \cite{ARTIOLI20202057, Dong2021GPUacceleratedDG} and the references therein, and we plan to investigate this issue in the future. 
\greenX{
\subsection{2nd experiment: Circular domain $\Omega^\sharp$}
We consider a two--dimensional test case of (\ref{The Stokes Problem}) in  $\Omega^\sharp$ to be a \greenX{disc} of radius $r=1$
%
 with 
 exact solution
$$
\mathbf{u}\left(x,y\right)=\left(u_x\left(x,y\right), u_y\left(y,x\right)\right), \quad p\left(x,y\right)=(y^2+x^2-r^2)^2 (2 y^2+x^2) exp(-k (y^2+ x^2))/2,
$$
where 
    $$u_x =  -y (y^2+x^2 - r^2) exp(-k (y^2+x^2)/2)  \text{ and   }  u_y=  x (y^2+x^2-r^2) exp(-k (y^2+x^2)/2), $$
$k=3\pi /2.$
 %
%
Note that the mean value of $p\left(x,y\right)$ over {$\Omega^\sharp$} vanishes by construction,
thus ensuring that the problem (\ref{The Stokes Problem}) is uniquely solvable as well as it is divergence free. Again, following  subsection \ref{22}, in the {spirit of arbitrarily shaped discontinuous Galerkin method on the boundary  
approach}, we consider the original domain {$\Omega^\sharp$} 
{
as it is  in Figure \ref{geometry}}'s second image.  A level set description of the geometry is possible through the function
\begin{equation}\label{lset}
\phi\left(x,y\right)=y^2+x^2 - r^2<0.
\end{equation}
\begin{table}[htbp]
\centering  
{\begin{tabular}{|c|c|c||c|c|} 
\toprule
  \greenX{Discretization}  &  \multicolumn{4}{c|}{\greenX{Errors and convergence rates: $\mathfrak{p}_1 / \mathfrak{p}_0$ polynomials case}} \\ 
\midrule 
\greenX{$h_{\max}$} &    \greenX{$\left\|\mathbf{u}-\mathbf{u}_h\right\|_{1,\Omega^\sharp}$} & \greenX{Conv. rate}
           & \greenX{$\left\|p-p_h\right\|_{\Omega^\sharp}$} & \greenX{Conv. rate} \\
\midrule 
\greenX{$0.62500625$}                      & \greenX{0.2091145} &   --                            & \greenX{0.1462374} &    --   \\
\greenX{$0.31250312$}                      & \greenX{0.1643200} & \greenX{0.3477844} & \greenX{0.0742179}  & \greenX{0.9784737} \\
\greenX{$0.15625156$}                      & \greenX{0.0842741} & \greenX{0.9633462} & \greenX{0.0469684}  & \greenX{0.6600764}\\
\greenX{$0.07812578$}                      & \greenX{0.0441366}&  \greenX{0.9331140} & \greenX{0.0265372} &  \greenX{0.8236700} \\
\greenX{$0.03906289$}                      & \greenX{0.0226625} & \greenX{0.9616639} & \greenX{0.0141321} &  \greenX{0.9090371}\\
\greenX{$0.01953144$}                      & \greenX{0.0114507} & \greenX{0.9848749} & \greenX{0.0072050} &  \greenX{0.9718974}\\
\greenX{$\quad\,0.00976572\quad\,$}   & \greenX{0.0057310} & \greenX{0.9985631} & \greenX{0.0035962} &  \greenX{1.0025094}\\
\midrule
\greenX{Mean:}       &          --            & {\greenX{0.968312}} &   --                           & {\greenX{0.873438}}
 \\
\bottomrule
\end{tabular}}
\caption{\greenX{Circular geometry: Errors and rates of convergence with respect to $H^1$-norm 
for the velocity and $L^2$-norm for  \greenX{the pressure, using}  {\greenX{$\mathfrak{p}_1 / \mathfrak{p}_0$ finite elements and with no ghost boundary elements area stabilization.}}}}
\label{table5p1p0C}
\end{table}
\begin{table}[htbp]
\centering  
{\begin{tabular}{|c|c|c||c|c|} 
\toprule
  \greenX{Discretization}  &  \multicolumn{4}{c|}{\greenX{Errors and convergence rates: $\mathfrak{p}_2 / \mathfrak{p}_1$ polynomials case}} \\ 
\midrule 
\greenX{$h_{\max}$} &    \greenX{$\left\|\mathbf{u}-\mathbf{u}_h\right\|_{1,\Omega^\sharp}$} & \greenX{Conv. rate}
           & \greenX{$\left\|p-p_h\right\|_{\Omega^\sharp}$} & \greenX{Conv. rate} \\
\midrule 
\greenX{$0.62500625$}                      & \greenX{0.0328117} &   --                            & \greenX{0.0145547} &    --   \\
\greenX{$0.31250312$}                      & \greenX{0.0319032} & \greenX{0.0405088} & \greenX{0.0123406}  & \greenX{0.2380803} \\
\greenX{$0.15625156$}                      & \greenX{0.0082061} & \greenX{1.9589227} & \greenX{0.0039063}  & \greenX{1.6595269}\\
\greenX{$0.07812578$}                      & \greenX{0.0020995}&  \greenX{1.9666160} & \greenX{0.0010171} &  \greenX{1.9412207} \\
\greenX{$0.03906289$}                      & \greenX{0.0005342} & \greenX{1.9746348} & \greenX{0.0002690} &  \greenX{1.9189042}\\
\greenX{$\quad\,0.01953144\quad\,$}   & \greenX{0.0001331} & \greenX{2.0038963} & \greenX{6.759e-05} &  \greenX{1.9927008}\\
\midrule
\greenX{Mean:}       &          --            & {\greenX{1.976017}} &   --                           & {\greenX{1.878088}}
 \\
\bottomrule
\end{tabular}}
\caption{\greenX{Circular geometry: Errors and rates of convergence with respect to $H^1$-norm 
for the velocity and $L^2$-norm for  \greenX{the pressure, using}  {\greenX{$\mathfrak{p}_2 / \mathfrak{p}_1$ finite elements and with no ghost boundary elements area stabilization.}}}}
\label{table5p2p1C}
\end{table}
\begin{table}[htbp]
\centering  
{\begin{tabular}{|c|c|c||c|c|} 
\toprule
  \greenX{Discretization}  &  \multicolumn{4}{c|}{\greenX{Errors and convergence rates: $\mathfrak{p}_3 / \mathfrak{p}_2$ polynomials case}} \\ 
\midrule 
\greenX{$h_{\max}$} &    \greenX{$\left\|\mathbf{u}-\mathbf{u}_h\right\|_{1,\Omega^\sharp}$} & \greenX{Conv. rate}
           & \greenX{$\left\|p-p_h\right\|_{\Omega^\sharp}$} & \greenX{Conv. rate} \\
\midrule 
\greenX{$0.62500625$}           & \greenX{0.0120753} &   --            & \greenX{0.0070916} &    --   \\
\greenX{$0.31250312$}         & \greenX{0.0044922} & \greenX{1.4265409}  &\greenX{0.0014351}  & \greenX{2.3048990} \\
\greenX{$0.15625156$}       & \greenX{0.0005162} & \greenX{3.1212517}  & \greenX{0.0001907}  & \greenX{2.9113670}\\
\greenX{$0.07812578$}     & \greenX{6.4007e-05} &  \greenX{3.0118057} & \greenX{2.7194e-05} &  \greenX{2.8103814} \\
\greenX{$\quad\,0.03906289\quad\,$}   & \greenX{8.3068e-06}  & \greenX{2.9458756} & \greenX{5.3857e-06} &  \greenX{2.3361018}\\
\midrule
\greenX{Mean:}       &          --       & {\greenX{2.62636850}} &   --      & {\greenX{2.5906872}
}
 \\
\bottomrule
\end{tabular}}
\caption{\greenX{Circular geometry: Errors and rates of convergence with respect to $H^1$-norm 
for the velocity and $L^2$-norm for  \greenX{the pressure, using}  {\greenX{$\mathfrak{p}_3 / \mathfrak{p}_2$ finite elements and with no ghost boundary elements area stabilization.}}}}
\label{table5p3p2C}
\end{table}
%
%
%
\begin{figure}
\centering
\begin{tabular}{cc}
\subfigure[ \greenX{$\|\mathbf{u}-\mathbf{u}_h\|_{1,\Omega^\sharp}$ error, conv.rates, and mean conv. rates}]{
\includegraphics[scale=0.50]{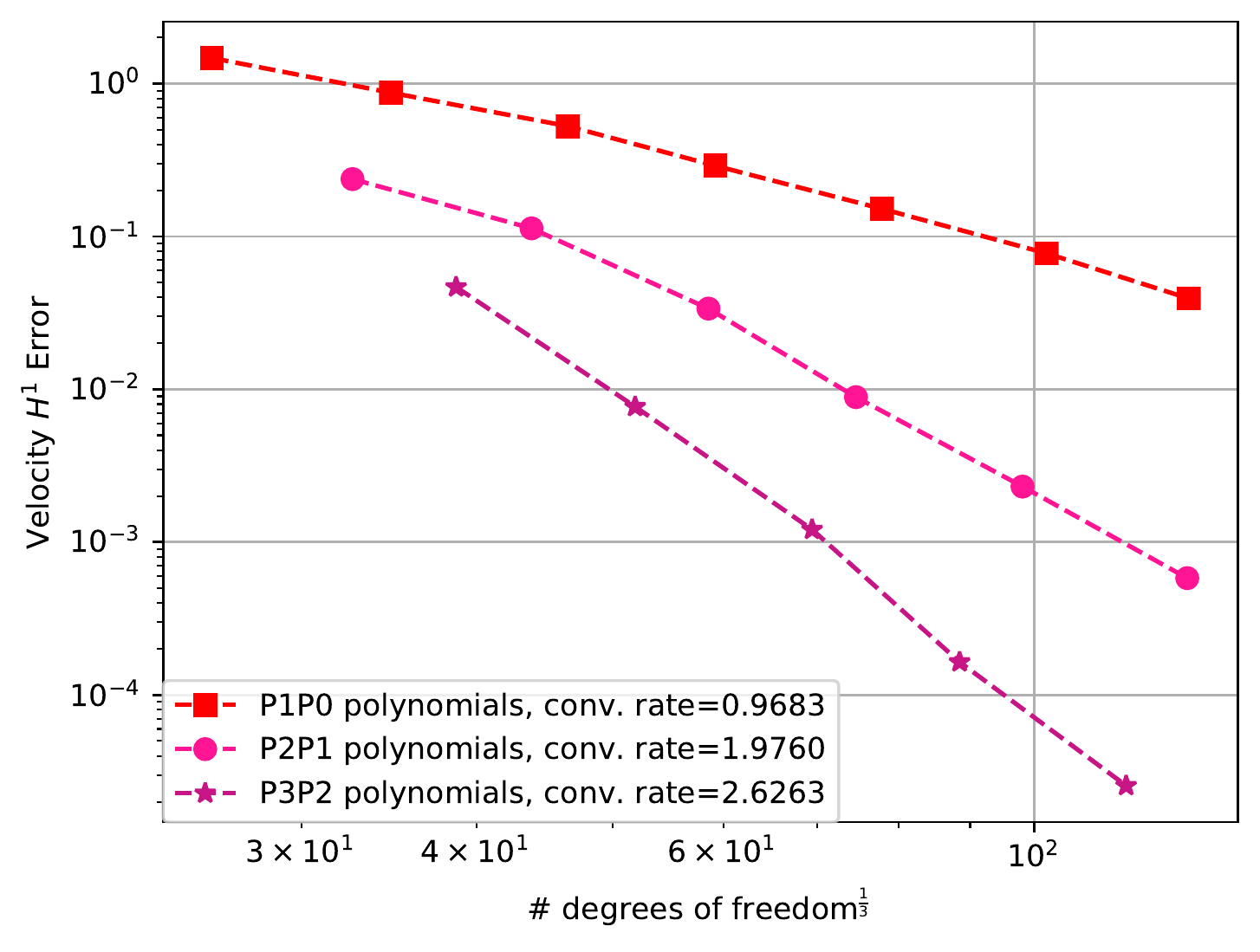}
}
\subfigure[\greenX{$\|p-p_h\|_{\Omega^\sharp}$  error, conv.rates, and mean conv. rates}]{
\includegraphics[scale=0.50]{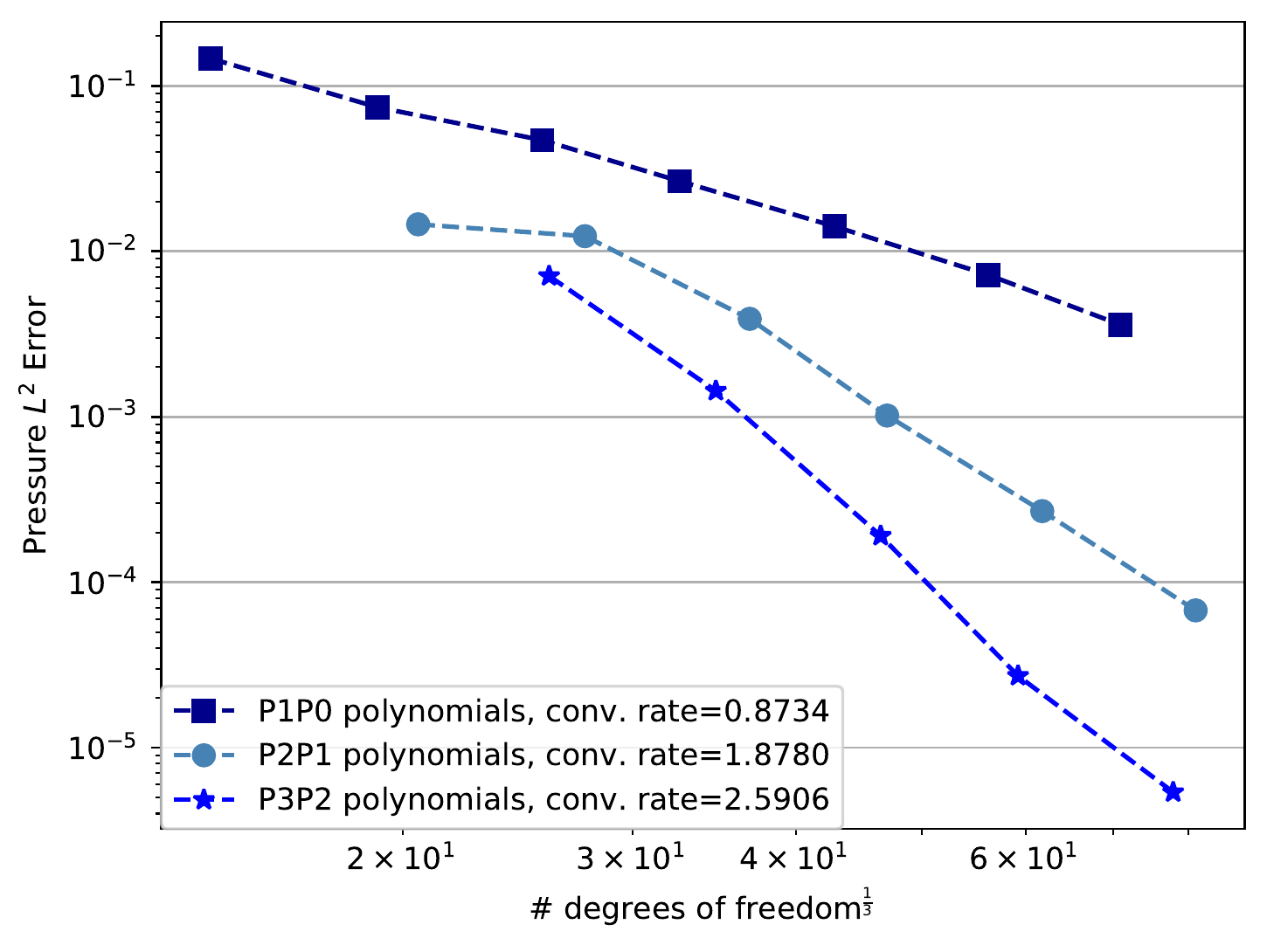}
}
\end{tabular}
\caption{\greenX{Circular geometry: Visualization of the $H^1$-norm velocity errors and $L^2$-norm pressure errors \greenX{with respect to the discretization size.}}}\label{sensitivity_V2_P3_C}
\end{figure}
We can easily confirm the results of Theorem \ref{order_revised} with the results as reported in Tables \ref{table5p1p0C},  \ref{table5p2p1C},  \ref{table5p3p2C} for $\mathfrak{p}=1,2,3$ respectively and for Taylor-Hood $\mathbf{u}$, $p$ pairs. The superiority of the higher order cases is clarified, with respect to the errors and to the convergence rates where we manage better results in a much coarser mesh. This can also be seen with a glance in the visualization of Figure \ref{sensitivity_V2_P3_C}.    

\blueX{Finally, we highlight that the convergence rates for both tests seems to stop in the denser meshes and especially in the highest order  polynomials rectangular geometry cases even if we notice good error results. This in our opinion happens due to the aforementioned way of integration on the polytopal boundary elements and we will  leave the correction for the future, and of course crucial role played the corner points since  in this last experiment --with the smoother circular geometry-- we noticed correction of the pressure suboptimal convergence rates.
}}
\section{Conclusion}
In this effort, we proposed and tested a 
discontinuous Galerkin method for the incompressible Stokes flow employing arbitrarily shaped elements. $hp$-version optimal order convergence is proved for higher order finite elements of 
$ \mathfrak{p}_i/\mathfrak{p}_{i-1}$ order for velocity and pressure fields. 
This method may prove valuable in 
applications where special emphasis is placed on the effective approximation of pressure, attaining much smaller  errors  in coarser meshes 
{and whenever geometrical morphings are taking place}. In fact, control over the error of the pressure field is among the most decisive points of difficulty for many methods. 
Numerical test experiments demonstrated the very good stability and accuracy properties of the method. The theoretical  convergence rates for the $H^1$-norm of the velocity and the $L^2$-norm of the pressure have been validated by our tests, even for the $\mathfrak{p}_3 / \mathfrak{p}_2 $ case\TODO
{. Significantly smaller errors have been noticed considering the comparison of the method under consideration, with the unfitted dG finite element method of \cite{AreKarKat22} and the results as reported in Table \ref{table33} and Table \ref{table44} for step $h=2^{-6}$. In particular $2.5518e-05$, $0.0003071$ for the presented approach, and $0.0002045$, $0.0003671$ 
 for a fully stabilized with efficient ghost penalty fictitious domain FEM and for the velocity and pressure respectively have been reported. The latter, after comparison, shows smaller errors for the pressure and one order of magnitude better errors for the velocity which validates the superiority of the proposed approach.} 
%
  %
In the present work, we focused on the static Stokes problem. Thus, the method seems very promising. Future work will extend our investigations to more general fluid mechanics problems, including time-dependent problems on complex 
 domains and/or nonlinearities, as well as, Navier-Stokes and fluid-structure interaction systems. Finally, future development would be a proper preconditioner and a reduced order modeling investigation.

\section*{Acknowledgments}
This project has received funding from 
``First Call for H.F.R.I. Research Projects to support Faculty members and Researchers and the procurement of high-cost research equipment'' grant 3270 and was supported by computational time granted from the National Infrastructures for Research and Technology S.A. (GRNET S.A.) in the National HPC facility - ARIS - under project ID pa190902. The author would like also to thank  Prof. Emmanuil Georgoulis 
from NTUA for valuable comments and inspiring ideas. We would like also to thank the contributors of the ngsolve-ngsxfem, \cite{Scho14,LeHePreWa21}.
\bibliographystyle{amsplain}
\bibliography{bibfile_sissa}

\end{document}